\documentclass{article}

\usepackage{graphics} 
\usepackage{epsfig}   
\usepackage{epstopdf}

\usepackage{amsmath,amssymb,amsfonts}
\usepackage{algorithmic}
\usepackage{graphicx}
\usepackage{booktabs}

\usepackage{subfigure}

\usepackage{balance}

\usepackage{amsthm}

\usepackage[margin=1in]{geometry}

\usepackage[font=small,labelfont=bf]{caption}
\usepackage[numbers]{natbib}
\usepackage{color}

\usepackage[hidelinks,urlcolor=black]{hyperref}

\newtheorem{thm}{Theorem}
\newtheorem{cor}{Corollary}

\newtheorem{defn}{Definition}
\newtheorem{lem}{Lemma}
\usepackage{enumerate}
\newtheorem{prop}{Proposition}
\def\QED{\mbox{\rule[0pt]{1.3ex}{1.3ex}}} 
\newenvironment{proof-of}[1]{{\it Proof of #1:\,}}{\hfill \QED \par}
\newtheorem{rem}{Remark}

\newcommand{\R}{{\mathbb R}}

\newcommand{\Rnn}{{\mathbb R}_{\ge 0}}
\newcommand{\Rp}{{\mathbb R}_{> 0}}

\newcommand{\cA}{{\mathcal A}}
\newcommand{\cB}{{\mathcal B}}
\newcommand{\cC}{{\mathcal C}}

\newcommand{\Hinf}{{\mathcal{H}_\infty}}

\newcommand{\cG}{{\mathcal G}}

\newcommand{\cI}{{\mathcal I}}

\newcommand{\cK}{{\mathcal K}}

\newcommand{\cL}{{\mathcal L}}
\newcommand{\cM}{{\mathcal M}}
\newcommand{\cN}{{\mathcal N}}

\newcommand{\cP}{{\mathcal P}}

\newcommand{\cU}{{\mathcal U}}

\newcommand{\cY}{{\mathcal Y}}

\newcommand{\diag}[1]{\textrm{diag}\{#1\}}

\renewcommand{\Re}{\mathrm{Re}}

\newcommand{\bfone}{\mathbf{1}}

\newcommand{\yang}[1]{{#1}}
\newcommand{\transpose}{{\mathsf T}}

\newcounter{tempEquationCounter}
\newcounter{thisEquationNumber}

\title{On the Existence of Block-Diagonal Solutions to Lyapunov and \texorpdfstring{$\Hinf$}{H infinity} Riccati Inequalities\thanks{This is an extended technical report. The main results have been accepted for publication as a technical note in the IEEE Transactions on Automatic Control. A. Sootla and A. Papachristodoulou are supported by the EPSRC Grant EP/M002454/1. Y. Zheng is supported in part by the Clarendon Scholarship, and in part by the Jason Hu Scholarship.}}

\author{Aivar Sootla, Yang Zheng, and Antonis Papachristodoulou
	\thanks{The authors are with the Department of Engineering Science, University of Oxford, Parks Road, Oxford, OX1 3PJ, U.K. e-mails: \{aivar.sootla, yang.zheng, antonis\}@eng.ox.ac.uk. }}

\begin{document}
	\maketitle
	
\begin{abstract}
	In this paper, we describe sufficient conditions when block-diagonal solutions to Lyapunov and $\Hinf$ Riccati inequalities exist. In order to derive our results, we define a new type of comparison systems, which are positive and are computed using the state-space matrices of the original (possibly nonpositive) systems. Computing the comparison system involves only the calculation of $\Hinf$ norms of its subsystems. We show that the stability of this comparison system implies the existence of block-diagonal solutions to Lyapunov and Riccati inequalities. Furthermore, our proof is constructive and the overall framework allows the computation of block-diagonal solutions to these matrix inequalities with linear algebra and linear programming. Numerical examples illustrate our theoretical results.
\end{abstract}

\section{Introduction}
Block-diagonal solutions to Lyapunov and Riccati inequalities are preferable in many control theoretic applications, e.g., structured model reduction (cf.~\cite{SandStrucRed}), decentralised control and analysis (cf.~\cite{siljak2011decentralized}). A class of systems that is known to admit block-diagonal solutions to these matrix inequalities is the class of positive systems (cf.~\cite{berman1994nonnegative, rantzer2015ejc}), which is one of the reasons why generalisations of this class of systems is an active area of research~\cite{ebihara2017h2, cacace2014stable}.

In this paper, we focus on a generalisation of positive systems based on diagonally dominant matrices since it is known that for this class of systems separable Lyapunov functions exist~\cite{hershkowitz1985lyapunov} and can be computed using linear programming~\cite{sootla2016existence}. Recently, the diagonal dominant approach was applied to block-partitioned matrices, which lead to conditions for the existence of block-diagonal solutions to Lyapunov inequalities~\cite{sootla2017blocksdd}. In this paper, we generalise the existence results from~\cite{sootla2017blocksdd} and derive conditions on the existence of block-diagonal solutions to $\Hinf$ Riccati inequalities. The main idea of the approach is to partition the state-space and compute a \emph{comparison system}, which is positive and its dimension is equal to the number of clusters in the state partition. Hence its dimension can be substantially smaller than the dimension of the original system. The computation of the comparison system reduces to the computation of $\Hinf$ norms of the systems, whose size is determined by the size of the individual clusters. We show that the stability of the comparison system implies stability of the original system (the converse is generally false), and guarantees the existence of block-diagonal solutions to Lyapunov and Riccati inequalities. The proof of this result is constructive and computing these solutions can be performed using linear algebra and linear programming methods.

Even though we took inspiration from the linear algebra literature, in our previous work~\cite{sootla2017blocksdd} we reconstructed and generalised some existing control theory results, in particular, the stability criteria in~\cite{cook1974stability}. Therefore, our comparison system approach is tightly related to previous work on comparison systems reported in~\cite{araki1975application, vidyasagar1981input} and more recently in~\cite{dashkovskiy2010small}. However, the computation of comparison systems in~\cite{araki1975application,dashkovskiy2010small} requires constructing Lyapunov/storage functions for individual systems in the network, and the overall procedure is typically non-convex. Our approach, on the other hand, is constructive and provides an algorithm to compute comparison systems \emph{without} Lyapunov function computation. Furthermore, in the context of linear systems, the existence and construction of block-diagonal solutions to Riccati and Lyapunov inequalities are not discussed before.

The rest of this paper is organised as follows. In Section~\ref{s:prel}, we cover some preliminaries of control theoretic tools, positive systems theory and define a new type of comparison systems. In Section~\ref{s:prop-comp-sys}, we derive sufficient conditions for the existence of block-diagonal solutions to $\Hinf$ Riccati inequalities. We illustrate our theoretical results in Section~\ref{s:exam}. Additional minor technical results and numerical simulations are available in Appendix.

\emph{Notation.} The minimal and maximal singular values of a matrix $A\in \R^{m\times n}$ are denoted by $\underline{\sigma}(A)$ and $\overline{\sigma}(A)$, respectively. For a matrix $A\in \R^{m\times n}$, $A^\transpose$ denotes its transpose.  The $\Hinf$ norm of an asymptotically stable transfer function $G(s)$ is computed as $\|G\|_\Hinf =\max_{w\in \R}\|G(\imath \omega)\|_2$, where $\imath$ is the imaginary unit and $\|A\|_2 = \overline{\sigma}(A)$. A positive semidefinite (resp., positive definite) matrix $X$ is denoted by $X\succeq 0$ (resp., $X\succ 0$). We denote the matrices with nonnegative (resp., positive) entries as $A \ge 0$ (resp., $A >0$). The nonnegative (resp., positive) orthant, \emph{i.e.}, the set of all vectors $x\ge0$ (resp., $x>0$) in $\R^n$, is denoted by $\Rnn^n$ (resp., $\Rp^n$). For matrices $X_{i j}, j=1, \ldots, n$ with compatible dimension, we use $X_{i, -i}$ to denote
\begin{gather}\label{xi-i}
		X_{i, -i} =\begin{bmatrix}
		X_{i, 1} & \cdots & X_{i, i-1} & X_{i, i+1} & \cdots & X_{i, n}
	\end{bmatrix}.
\end{gather}
Finally, a block-diagonal matrix with matrices $A_i, i = 1, \ldots, n,$ on its diagonal is denoted by $\diag{A_1, \dots, A_n}$, \emph{i.e.},
\[
    \diag{A_1, \dots, A_n} = \begin{bmatrix} A_1 & & \\ & \ddots & \\ & & A_n \end{bmatrix}.
\]

\section{Preliminaries}~\label{s:prel}
In this section, we present some preliminaries on positive systems and introduce a new comparison system that is positive by definition.
\subsection{Control Theoretic and Positive Systems Tools}
In this paper, we study linear time-invariant systems
\begin{gather}
\label{eq:system-full}
\begin{aligned}
\dot{x}(t)&=A x(t) + B u(t),\\
y(t) &= C x(t) + D u(t),
\end{aligned}
\end{gather}
where $A \in \R^{N \times N}$, $B \in \R^{N \times N_i}$, $C \in \R^{N_o \times N}$, and $D\in \R^{N_o\times N_i}$. System~\eqref{eq:system-full} is asymptotically stable if and only if $A$ is a Hurwitz matrix, \emph{i.e.}, all its eigenvalues have negative real parts~\cite{ZDG} or equivalently there exists a positive definite matrix $P\succ 0$ such that
\begin{gather}\label{eq:lyap_lmi}
P A + A^\transpose P\prec 0.
\end{gather}
Using the linear matrix inequality (LMI)~\eqref{eq:lyap_lmi} (usually called Lyapunov inequality), one can define a Lyapunov function of the form $V(x)=x(t)^\transpose P x(t)$ for system~\eqref{eq:system-full} with $u(t) = 0$. In the context of input-output behaviour, \emph{dissipativity} is considered as a typical analysis tool and in particular $\Hinf$ analysis is enabled by the \emph{Bounded Real Lemma}~\cite{ZDG}.

\begin{prop}\label{prop:brl}
	Consider a system~\eqref{eq:system-full} where  $A$ is Hurwitz. We have
\begin{enumerate}[(a)]
  \item $\|C (s I - A)^{-1} B + D\|_\Hinf < \delta$ if and only if $\overline{\sigma}(D) < \delta$ and there exists $P \succ 0$ such that the following Riccati inequality holds
	\begin{equation}\label{ric-hinf}
	P A + A^\transpose P  + C^\transpose C  - (P B +C^\transpose D) (D^\transpose D  -\delta^{2} I)^{-1} (D^\transpose C  + B^\transpose P)  \prec  0.
	\end{equation}
  \item if $(C,A)$ is observable, then $ \|C (s I - A)^{-1} B+D\|_\Hinf < \delta$ implies that there exists $P\succ 0$ such that~\eqref{ric-hinf} holds with equality instead of inequality.
\end{enumerate}
\end{prop}

We refer the interested reader to Corollary 13.24 in~\cite{ZDG} for a detailed proof. Note that the converse to the point b) holds only with additional spectral constraints on the solution $P$ and the system matrices $A$, $B$, $C$, $D$.
As the reader may notice both analysis methods rely on LMIs
with a generally dense positive definite matrix $P \succ 0$. In some cases, analysis can be significantly simplified using vector inequalities, which happens in the case of \emph{positive systems}. A system is called (internally) positive, if for any nonnegative control signal $u(t)$, and any nonnegative initial condition $x(0) = x_0$, the state $x(t)$ and the output $y(t)$ remain nonnegative. In order to avoid confusion, we will use a different notation for positive systems, namely:
\begin{gather}
\begin{aligned}
\dot \xi  &= F \xi + G \upsilon,\\
\nu &= H \xi +J \upsilon,
\end{aligned} \label{eq:system-comparison}
\end{gather}
where $F \in \R^{n \times n}$, $G \in \R^{n \times n_i}$, $H \in \R^{n_o \times n}$, and $J \in \R^{n_o \times n_i}$. Internally positive systems are fully characterised by conditions on the matrices $F$, $G$, $H$ and $J$: System~\eqref{eq:system-comparison} is internally positive
if and only if the matrices $G$, $H$, $J$ are nonnegative (all their entries $G_{i l}$, $H_{k j}$ , $J_{k l}$ are nonnegative) and \emph{the matrix $F$ is Metzler} (all its off-diagonal elements $F_{i j}$ for $i\ne j$ are nonnegative)~\cite{kaczorek2001externally}. In terms of stability and $\Hinf$ analyses, two well-known results, which can be found in~\cite{rantzer2015ejc, rantzer2016kalman, fan1958topological, varga1976recurring}, showcase the simplification.
\begin{prop}\label{prop:pos-stab}
	Consider a Metzler matrix $F$. Then the following statements are equivalent:
\vspace{-1mm}
    \begin{enumerate}[(a)]
	\item  $F$ is Hurwitz;
    \item There exists $d \in \Rp^n$ such that $-F d \in \Rp^n$;
	\item There exists $e\in \Rp^n$ such that $- F^\transpose e  \in \Rp^n$;
	\item There exists a diagonal $P\succ 0$ such that $P F + F^\transpose P \prec 0$.
    \end{enumerate}
\end{prop}

\begin{prop}\label{prop:pos-hinf}
	Consider system~\eqref{eq:system-comparison} where $G$, $H$, $J$ are nonnegative matrices, while $F$ is a Hurwitz and Metzler matrix. Then the following statements are equivalent for a scalar $\delta$:
\vspace{-1mm}
	 \begin{enumerate}[(a)]
\item $\| H(sI - F)^{-1} G + J\|_\Hinf<\delta$;
\item $\delta > \overline{\sigma}(J)$ and there exists a diagonal matrix $P\succ 0$ such that
	\begin{equation} \label{ric-ineq-pos}
	P F + F^\transpose P  + H^\transpose H- ( P G + H^\transpose J) (J^\transpose J - \delta^{2} I )^{-1} (J^\transpose H + G^\transpose P)  \prec 0;
	\end{equation}
	\item $\overline{\sigma}(-H F^{-1} G + J) < \delta$;
\item  There exist vectors $d, e \in \Rp^n$, $g\in \Rnn^{n_o}$, $f\in \Rp^{n_i}$ such that
	\begin{align}
	& F d + G f < 0,           && H d + J f \le g,  \label{cond-pos-hinf-1}\\
	& F^\transpose e + H^\transpose g < 0, && G^\transpose e + J^\transpose g < \delta^2 f.  \label{cond-pos-hinf-2}
	\end{align}
   \end{enumerate}
\end{prop}

\vspace{1mm}

Note that condition (d) can be obtained from condition (1.4) in Theorem~1 in~\cite{rantzer2016kalman} in the case of strict inequalities. If a certain $g_k =0$, then the whole row of the matrices $H$ and $J$ is equal to zero. Therefore, without loss of generality, we can assume that $g$ is a positive vector.
\subsection{Definition of a Comparison System}
We say that a matrix $A\in\R^{N\times N }$ has \emph{$\alpha=\{k_1, \dots, k_n\}$-partitioning} with $N = \sum_{i = 1}^n k_i$, if the matrix $A$ is written with $A_{i j}\in\R^{k_i\times k_j}$ as follows
\[
A = \begin{bmatrix}
A_{1 1}   & \dots   & A_{1 n} \\
\vdots    & \ddots  & \vdots  \\
A_{n 1}   & \dots   & A_{n n}
\end{bmatrix}.
\]
The matrix $A\in\R^{N\times N}$ is \emph{$\alpha$-diagonal} if it is $\alpha$-partitioned and $A_{i j} = 0$ for $i\ne j$. The matrix $A$ is \emph{$\alpha$-diagonally stable}, if there exists an $\alpha$-diagonal positive definite $P\in\R^{N\times N}$ satisfying~\eqref{eq:lyap_lmi}.

Our goal is to perform analysis of partitioned systems~\eqref{eq:system-full} using only meta-information about the system, such as, the norms of the blocks $A_{i j}$, $B_{i j}$, $C_{i j}$, where the indices of $B_{i j}$ take the values $i = 1, \dots,n$, $j = 1,\dots, n_i$, while the indices of $C_{i j}$ take the values $i = 1,\dots, n_o$, $j = 1,\dots,n$. Using this meta-information, we define an internally positive system~\eqref{eq:system-comparison} with $n\le N$, $n_i \le N_i$, $n_o \le N_o$ that we will call a \emph{comparison system}. If we take $n_i< N_i$ and $n_o < N_o$ it means we lump some of the inputs and outputs into one signal. The main question is how to choose $F$, $G$, $H$, $J$ so that analysis of the comparison system yields meaningful properties of system~\eqref{eq:system-full}. We first present a new comparison matrix inspired by~\cite{sootla2017blocksdd, xiang1998weak}.
\begin{defn}\label{def:block-comp-1}
	Given an $\alpha$-partitioned matrix $A$ with Hurwitz $A_{i i}$, we define a comparison matrix $\cM^\alpha$ as follows:
	\begin{equation}
	\cM^\alpha_{i j} = \begin{cases} -1 &  \text{if }i = j, \\
	\|(s I - A_{i i})^{-1}A_{i j}\|_\Hinf & \textrm{otherwise}.
	\end{cases} \label{block-comparison-1}
	\end{equation}
\end{defn}
Definition~\ref{def:block-comp-1} is in the spirit of the generalisations of scaled diagonally dominant matrices discussed in~\cite{polman1987incomplete, feingold1962block, xiang1998weak} and {is} a direct generalisation of the definition in~\cite{sootla2017blocksdd}. In order to streamline the presentation we discuss the connection to~\cite{sootla2017blocksdd} in the Appendix.

Based on the comparison matrix $\cM^\alpha(A)$, we define the comparison system as follows:
\begin{equation}
\begin{aligned}
F = \cM^\alpha(A),\, G_{i l} &= \| (s I -A_{i i})^{-1} B_{i l} \|_\Hinf, \\
H_{k j} = \| C_{k j} \|_2, J_{k l} &= \|D_{k l}\|_2, \label{comp-sys-hard}
\end{aligned}
\end{equation}
for $i, j = 1,\dots, n$, $k = 1,\dots, n_o$, $l = 1,\dots, n_i$.

\section{Block-diagonal Solutions to the \texorpdfstring{$\Hinf$}{H infinity} Riccati Inequalities}\label{s:prop-comp-sys}
Our main theoretical result states that the $\Hinf$ norm of a system is bounded above by the $\Hinf$ norm of its comparison system. 

\begin{thm}\label{prop:comp-ma-hinf}
	Consider system~\eqref{eq:system-full} with the comparison system~\eqref{eq:system-comparison} defined in~\eqref{comp-sys-hard}. If $F$ is Hurwitz and $\|H (s I-F)^{-1}G +J\|_\Hinf<\delta$,
	then $\|C(s I - A)^{-1} B +D \|_\Hinf < \delta$ and there exist $P_i\succ 0$ such that~\eqref{ric-hinf} holds with $P = \diag{P_1, \dots, P_n}$.
\end{thm}

Besides the norm estimation, Theorem~\ref{prop:comp-ma-hinf} provides a sufficient condition for the existence of block-diagonal solutions to $\Hinf$ Riccati inequality~\eqref{ric-hinf}. The proof of Theorem~\ref{prop:comp-ma-hinf} is constructive and will illustrate how the block-diagonal $P$ can be constructed using linear programming and linear algebra. It is also straightforward to show that stability of $\cM^\alpha(A)$ implies the existence of a block-diagonal solution to the Lyapunov inequality~\eqref{eq:lyap_lmi}. Again these solutions can be explicitly constructed.
To prove Theorem~\ref{prop:comp-ma-hinf}, we first present the following lemma.
\begin{lem}\label{lem:ric-dist}
	Consider system~\eqref{eq:system-full} with the comparison system~\eqref{eq:system-comparison} defined in~\eqref{comp-sys-hard}. Let 
	\begin{align*}
	 \cK_i &= \left\{k \in [1,\dots, n_o] \bigl| \,\, \|C_{k i}\|_2 \ne 0\right\}, \\
	 \cL_i &= \left\{l \in [1,\dots, n_i] \bigl| \,\, \|B_{i l}\|_2 \ne 0\right\}, \\
	 \cI_i &= \left\{j \in [1,\dots, n] \bigl| \,\, \|A_{i j}\|_2 \ne 0, j\ne i \right\}.
	\end{align*}
	If $F$ is Hurwitz and
	\[\|H (s I-F)^{-1}G +J\|_\Hinf<\delta,\]
	then there exist $P_i\succ 0$, $\Xi_{k i}, \Gamma_{i l}, \Phi_{i j}, \Upsilon_{k l}, \Lambda_{k l}\succeq 0$ such that:
	\begin{subequations}
		\label{ric:dis-dis-test}
		\begin{align}
		\label{ric:dis-dis-test-relax_a}     &\forall i: P_i A_{i i} + A_{i i}^\transpose P_i + \sum_{j = 1, 	j\ne i}^n \Phi_{j i} +  \sum_{k\in \cK_i}C_{k i}^\transpose \Xi_{k i}^{-1} C_{k i} + P_i\left(\sum_{j \in \cI_i} A_{i j} \Phi_{i j}^{-1} A_{i j}^\transpose  + \sum_{l \in \cL_i} B_{i l} \Gamma_{i l}^{-1} B_{i l}^\transpose\right) P_i \prec 0, \\
		&\forall k, l: \begin{bmatrix}
		\Upsilon_{k l} & -D_{k l}^\transpose \\
		-D_{k l} & \Lambda_{k l}
		\end{bmatrix} \succeq 0, \label{ric:dis-dis-test-relax_b}\\
		\label{ric:dis-dis-test-relax_c}&\forall k: \sum\limits_{i = 1}^n \Xi_{k i} + \sum\limits_{l = 1}^{n_i} \Lambda_{k l} \preceq I, \\
		\label{ric:dis-dis-test-relax_d}&\forall l:  \sum\limits_{i = 1}^n \Gamma_{i l} + \sum\limits_{k = 1}^{n_o} \Upsilon_{k l} \prec \delta^2 I.
		\end{align}
	\end{subequations}
\end{lem}
\begin{proof}
	According to Proposition~\ref{prop:pos-hinf} there exist positive scalars $e_i$, $d_i$, $g_i$, $f_i$ such that~(\ref{cond-pos-hinf-1}, \ref{cond-pos-hinf-2}) hold. Let $\eta_{k j} = H_{k j} d_j/ g_k$, $\gamma_{i l} = G_{i l} e_i/ f_l$, $\phi_{i j} = |F_{i j}| e_i/d_j$, $\lambda_{k l} = J_{k l}f_l/g_k I$, and $\mu_{k l} = J_{k l} g_k/f_l I$. Since we need to show the existence of a feasible solution to~\eqref{ric:dis-dis-test}, assume that $\Phi_{i j}=\phi_{i j} I$, $\Xi_{k i} = \eta_{k i} I$, $\Gamma_{i l} = \gamma_{i l}I$, $\Upsilon_{k l} =\mu_{k l} I$, and  $\Lambda_{k l} =\lambda_{k l} I$, where lower case variables denote scalars. Note that by construction the matrices $\Phi_{i j}$, $\Xi_{k i}$, $\Gamma_{i l}$, $\Upsilon_{k l}$, $\Lambda_{k l}$ are  either zero or positive definite. Let
	\begin{align*}
	&\begin{aligned}
	B_i &= \begin{bmatrix} B_{i 1} & \cdots & B_{i n_i}
	\end{bmatrix},& C_{i}^\transpose &=\begin{bmatrix}
	C_{1 i}^\transpose &\cdots&C_{n_o i}^\transpose
	\end{bmatrix}, \\
	\Upsilon_k &= \diag{\Upsilon_{k 1}, \dots, \Upsilon_{k n_i}}, &  \Lambda_l &= \diag{\Lambda_{1 l}, \dots, \Lambda_{n_o l}}, \\
	\Gamma_i &= \diag{\Gamma_{i 1}, \dots, \Gamma_{i n_i}}, &  \Xi_i &= \diag{\Xi_{1 i}, \dots, \Xi_{n_o i}},
		\end{aligned}
	\end{align*}
{and let $A_{i,-i}$, $\Phi_{i,-i}$ be defined as in~\eqref{xi-i}.} We define the matrices $\tilde \Phi_{i,-i}$ and $\tilde A_{i,-i}$ by removing all blocks from $\Phi_{i,-i}$ and $A_{i,-i}$ with $j\not \in \cI_i$. Similarly we define the matrices $\tilde B_i$, $\tilde \Gamma_i$.
	
First, we prove the following $\Hinf$ norm bound for all $i$:
\begin{equation} \label{dis_actual}
	\|(s I - A_{i i})^{-1} \begin{bmatrix} \tilde A_{i, -i} \tilde\Phi_{i,-i}^{-1/2} & \tilde B_i \tilde \Gamma_i^{-1/2}\end{bmatrix}\|_\Hinf^2 < \phi_{i i}^{-1}.
	\end{equation}
	
This can be shown by recalling~\eqref{cond-pos-hinf-1}:
	\begin{equation*}
	\begin{aligned}
	& \|(s I - A_{i i})^{-1} \begin{bmatrix} \tilde A_{i, -i} \tilde\Phi_{i,-i}^{-1/2} & \tilde B_i \tilde \Gamma_i^{-1/2}\end{bmatrix}\|_\Hinf^2 \\
	=&    \max_{s \in \imath \R} \overline{\sigma}\Bigl(\sum_{j\in \cI_i} (s I - A_{i i})^{-1} A_{i j} A_{i j}^\transpose(s I - A_{i i})^{-\transpose} /\phi_{i j}   + \sum_{l\in \cL_i} (s I - A_{i i})^{-1} B_{i l} B_{i l}^\transpose(s I - A_{i i})^{-\transpose} /\gamma_{i l}\Bigr) \\
	\le & \sum_{j \in \cI_i} \max_{s \in \imath \R }\overline{\sigma}\left((s I - A_{i i})^{-1} A_{i j} A_{i j}^\transpose (s I - A_{i i})^{-\transpose}\right)/\phi_{i j}  + \sum_{l \in \cL_i} \max_{s \in \imath \R }\overline{\sigma}\left((s I - A_{i i})^{-1} B_{i l} B_{i l}^\transpose (s I - A_{i i})^{-\transpose}\right)/\gamma_{i l}
	\\
	=&    \sum_{j \in \cI_i} \max_{s \in \imath \R}(\overline{\sigma}((sI - A_{i i})^{-1}A_{i j}))^2/\phi_{i j}    +\sum_{l \in \cL_i} \max_{s \in \imath \R}(\overline{\sigma}((s I - A_{i i})^{-1}B_{i l}))^2/\gamma_{i l}\\
	= &\sum_{j\in \cI_i} F_{i j} d_j/e_i + \sum_{l\in\cL_i} G_{i l} f_l/e_i\\
	<&  d_i /e_i
	= \phi_{i i}^{-1}.
	\end{aligned}
	\end{equation*}
	
Since $(I,A)$ is always observable, the bounded real lemma (Proposition~\ref{prop:brl}) and~\eqref{dis_actual} imply that for all $i$ there exist $P_i\succ 0$ such that
	\begin{equation}
 P_i A_{i i} + A_{i i}^\transpose P_i+ \phi_{i i} I 
 P_i\left(\sum_{j \in \cI_i}A_{i j} \phi_{i j}^{-1}A_{i j}^\transpose + \sum_{l \in \cL_i} B_{i l}\gamma_{i l}^{-1} B_{i l}^\transpose\right) P_i = 0. \label{ric_actual}
	\end{equation}
	
Next, considering~\eqref{cond-pos-hinf-2}, we have:
	\begin{equation}
	\label{cond:ric-phi}
	\begin{aligned}	\phi_{i i} I = \frac{e_i}{d_i} I \succ \sum\limits_{j = 1, j \ne i}^n \frac{F_{j i} e_j}{d_i}I + \sum\limits_{k =1}^{n_o} \frac{H_{k i} g_k}{d_i}I
	\succeq &\sum\limits_{j = 1, j \ne i}^n \phi_{j i} I +\sum_{k \in \cK_i}  \frac{C_{k i}^\transpose C_{k i} g_k}{H_{k i} d_i} \\
	=&    \sum\limits_{j = 1, j \ne i}^n \phi_{j i} I  +  \sum_{k \in \cK_i}   \frac{C_{k i}^\transpose C_{k i}}{\eta_{k i}}.
	\end{aligned}
	\end{equation}

	Substituting~\eqref{cond:ric-phi} into~\eqref{ric_actual} leads to
	\begin{equation*}
	P_i A_{i i} + A_{i i}^\transpose P_i   + \sum_{j = 1, j\ne i}^n \phi_{j i} I +  \sum_{k \in \cK_i} C_{k i}^\transpose \eta_{k i}^{-1} C_{k i} +	P_i\left(\sum_{j \in \cI_i} A_{i j} \phi_{i j}^{-1} A_{i j}^\transpose  + \sum_{l \in \cL_i} B_{i l} \gamma_{i l}^{-1} B_{i l}^\transpose\right) P_i\prec 0.
	\end{equation*}
    
    This means~\eqref{ric:dis-dis-test-relax_a} holds with $\Phi_{i j}=\phi_{i j} I$, $\Xi_{k i} = \eta_{k i} I$, $\Gamma_{i l} = \gamma_{i l}I$. The coupling constraints~\eqref{ric:dis-dis-test-relax_b}-\eqref{ric:dis-dis-test-relax_d} are straightforward.
\end{proof}

We are now ready to prove the main result of this note.

\begin{proof-of}{Theorem~\ref{prop:comp-ma-hinf}}
Using Schur's complement we can obtain the following LMI for a sufficiently small $\varepsilon_1>0$:
	\begin{gather*}
	\begin{bmatrix}
	P_i A_{i i} + A_{i i}^\transpose P_i  +\varepsilon_1 I  + \sum\limits_{j = 1, j\ne i}^n \Phi_{j i} & P_i A_{i,-i} & P_i B_i & C_i^\transpose \\
	\ast                                                                  & -\Phi_{i,-i}    & 0       & 0 \\
	\ast                                                                  & 0            & -\Gamma_i& 0 \\
	\ast                                                                  & 0            & 0       & -\Xi_i
	\end{bmatrix} \preceq 0.
	\end{gather*}
	The inequality is not strict, since some of the columns and rows can be equal to zero. By rearranging the matrices so that left most corner lies on the $i$-th diagonal entry, summing the resulting matrices for all $i$, we get
	\begin{gather} \label{lmi-theorem1-1}
	\begin{bmatrix}
	P A + A^\transpose P  +\varepsilon_1 I & P B       & C^\transpose \\
	B^\transpose P       & -\sum_{i =1}^{n} \Gamma_i & 0 \\
	C              & 0         & - \sum_{i =1}^{n} \Xi_i
	\end{bmatrix}\preceq 0.
	\end{gather}
	Multiplying the resulting LMI from the right with
	\[\begin{bmatrix}
	I & 0 & 0 \\
	0 & I & 0 \\
	0 & D & I
	\end{bmatrix}\]
	and from the left with its transpose results in:
	\begin{gather*}
	\begin{bmatrix}
	P A + A^\transpose P  +\varepsilon_1 I & \begin{bmatrix}
	P B +C^\transpose D      & C^\transpose
	\end{bmatrix} \\
	\begin{bmatrix}
	B^\transpose P + D^\transpose C \\
	C
	\end{bmatrix}       & Y
	\end{bmatrix}\preceq 0,
	\end{gather*}
	where
	\begin{gather*}
	Y = \begin{bmatrix}
	I&  D^\transpose\\ 0 & I
	\end{bmatrix}\begin{bmatrix}
	-\sum_{i=1}^{n} \Gamma_i & 0 \\
	0 & -\sum_{i=1}^{n} \Xi_i
	\end{bmatrix}\begin{bmatrix}
	I & 0 \\ D & I
	\end{bmatrix}.
	\end{gather*}
	We can complete the proof if we show that for a small positive $\varepsilon_2$ {the following inequality holds:}
	\begin{gather}\label{cond_y}
	Y \succeq \begin{bmatrix}
	D^\transpose D - (\delta^2-\varepsilon_2) I & 0 \\ 0 & - I\end{bmatrix}.
	\end{gather}
	In this case, we get
	\begin{gather*}
	\begin{bmatrix}
	P A + A^\transpose P +\varepsilon_1 I & P B +C^\transpose D      & C^\transpose \\
	B^\transpose P + D^\transpose C  & D^\transpose D - (\delta^2-\varepsilon_2) I & 0 \\
	C                                & 0                           & -I
	\end{bmatrix}\preceq 0,
	\end{gather*}
	and $\delta > \overline{\sigma}(D)$. Setting $\varepsilon_2 = 0$, application of the Schur complement and Proposition~\ref{prop:brl} will complete the proof.
	Multiplying condition~\eqref{cond_y} from the right with \[\begin{bmatrix} I & 0\\ - D & I \end{bmatrix}\] and from the left with its transpose we obtain an equivalent condition:
\begin{gather*}
\begin{bmatrix}
-\sum_{i=1}^{n} \Gamma_i & 0 \\
0 & -\sum_{i=1}^{n} \Xi_i
\end{bmatrix} \succeq \begin{bmatrix}
 -(\delta^2-\varepsilon_2) I & D^\transpose \\
D & -I
\end{bmatrix}.
\end{gather*}
According to \eqref{ric:dis-dis-test-relax_c} and~\eqref{ric:dis-dis-test-relax_d}, for a small $\varepsilon_2>0$ we have
\[
\begin{aligned}
&-\sum_{i =1}^{n} \Gamma_i+(\delta^2-\varepsilon_2) I \succeq \sum_{k=1}^{n_o} \Upsilon_k,~~\\
&-\sum_{i =1}^{n} \Xi_i+I \succeq \sum_{l=1}^{n_i} \Lambda_l,
\end{aligned}
\]
and we need to show that 	
	\begin{gather*}
\begin{bmatrix}
\sum_{k=1}^{n_o} \Upsilon_k & -D^\transpose \\
-D & \sum_{l=1}^{n_i} \Lambda_l
\end{bmatrix} \succeq 0.
\end{gather*}
The latter LMI follows from composing and adding the LMIs in~\eqref{ric:dis-dis-test-relax_b} in an appropriate manner.
\end{proof-of}

\begin{rem}[Construction of Lyapunov Functions]
The proof of Theorem~\ref{prop:comp-ma-hinf} provides a {constructive} way to find a block-diagonal solution $P = \diag{P_1, \dots, P_n} $ to the $\Hinf$ Riccati inequality, provided the comparison system is stable:
	
  \emph{Step 1:} Compute the comparison system~\eqref{comp-sys-hard} and its $\Hinf$ norm;

  \emph{Step 2:} Solve the linear program~\eqref{cond-pos-hinf-1} and \eqref{cond-pos-hinf-2};

  \emph{Step 3:} Solve the individual Riccati equation~\eqref{ric_actual} to find $P_i$.

This procedure clearly shows that a block-diagonal solution $P = \diag{P_1, \dots, P_n}$ can be constructed using linear programs and linear algebra if the comparison system is stable. {This requires less memory and computational power than solving an SDP (\emph{e.g.},~\eqref{ric-hinf} or~\eqref{ric:dis-dis-test}), which will be demonstrated using numerical examples in Section~\ref{ss:time}.}
\end{rem}

\begin{rem}[Small-gain Interpretation]
If $B$, $C$, $D$ are zero matrices, condition~\eqref{ric:dis-dis-test-relax_a} leads to the following small-gain type condition: The matrix $A$ is $\alpha$-diagonally stable if there exist $\Phi_{j i} \succeq 0$ such that
\begin{gather*}
\|\Phi_{i i}^{1/2} (s I - A_{i i})^{-1} \tilde A_{i,-i} \tilde \Phi_{i,-i}^{-1/2}\|_\Hinf < 1, \forall i,
\end{gather*}
where $A_{i,-i}$, $\Phi_{i,-i}$ are defined as in~\eqref{xi-i}, $\tilde A_{i,-i}$ $\tilde \Phi_{i,-i}$ are obtained by removing all blocks from $\Phi_{i,-i}$, $A_{i,-i}$ with $j\not \in \cI_i$ (i.e., zero blocks), and $\Phi_{i i} \succ \sum_{j = 1, 	j\ne i}^n \Phi_{j i}$. This interpretation shows that our conditions can take implicitly into account scaling factors, which are common in small-gain type results.
\end{rem}
\begin{rem}[SDP Conditions~\eqref{ric:dis-dis-test}]
These conditions are clearly less conservative than the comparison system approach, as in the proof of Lemma~\ref{lem:ric-dist} we used
several relaxations to obtain the SDP conditions from LP conditions~\eqref{cond-pos-hinf-1} and~\eqref{cond-pos-hinf-2}. On the other hand, solving~\eqref{ric:dis-dis-test} is more computationally expensive as it involves SDP constraints. We note that conditions~\eqref{ric:dis-dis-test} are of lower dimensions than~\eqref{ric-hinf}, which can be potentially taken advantage of. However, how to exploit the structure in~\eqref{ric:dis-dis-test} is not trivial and requires further research.
\end{rem}

We conclude this section by presenting some additional results for positive systems, which are straightforward to show using the proof of Theorem~\ref{prop:comp-ma-hinf}.

\begin{cor}
Consider a stable positive system~\eqref{eq:system-comparison}, then
the following statements hold:
\begin{enumerate}[(a)]
\item The solution to the Riccati inequality~\eqref{ric-ineq-pos}  can be computed as $P= \diag{e_1/d_1, \dots, e_n/d_n}$, where $e_i$ and $d_i$ satisfy~(\ref{cond-pos-hinf-1},~\ref{cond-pos-hinf-2});
\item Given positive vectors $d$, $e$ satisfying $F d <0$ and $e > - F^{-\transpose} H^\transpose H d$, we have $F^\transpose Q + Q F + H^\transpose H \prec 0$ with $Q= \diag{e_1/d_1, \dots, e_n/d_n}$;
\item Given positive vectors $d$, $e$ satisfying $F^\transpose d <0$ and $e > -F^{-1} G G^\transpose d$, we have $F P + P F^\transpose + G G^\transpose \prec 0$ with $P= \diag{e_1/d_1, \dots, e_n/d_n}$.
\end{enumerate}
\end{cor}

Note that in points (b) and (c), the computed solutions are not generally minimum trace solutions to these Lyapunov inequalities.

\section{Numerical Examples} \label{s:exam}
\subsection{\texorpdfstring{$\Hinf$}{H infinity} Performance Analysis with \texorpdfstring{$\alpha = \{2, 3\}$}{a = 2, 3}}

\begin{table}[b]
	\caption{Relative Errors of Sparse Estimates with Respect to $\Hinf$ norm } \label{tab:errors}
	\renewcommand{\arraystretch}{1.0}	
	\centering
	\begin{tabular}{ccccc}
\toprule
		&  $A^1$, $C^1$   &  $A^1$, $C^2$   & $A^2$, $C^1$   & $A^2$, $C^2$   \\
		\hline 
		$\dfrac{\delta_{\rm bd}}{\| C^i (s I - A^j)^{-1} B\|_\Hinf}$                &   $1$      &   $1.2914$   & $1$      & $1.7652$ \\[15pt]
		$\dfrac{\delta_{\rm sdp}}{\| C^i (s I - A^j)^{-1} B\|_\Hinf}$               &   $1$      &   $1.9448$   & $1$      & $4.6133$ \\[15pt]
		$\dfrac{\delta_{\rm scal}}{\| C^i (s I - A^j)^{-1} B\|_\Hinf}$              &   $1.0747$ &   $2.0948$   & $1.2545$ & $5.9495$ \\[15pt]
		$\dfrac{\| H (s I - F)^{-1} G\|_\Hinf}{\| C^i (s I - A^j)^{-1} B\|_\Hinf}$  &   $1.2947$ &   $3.3578$   & $1.6210$ & $12.4408$ \\
\bottomrule 	
\end{tabular}
\end{table}

Since the comparison system is computed using linear algebraic tools in a completely distributed manner, the scalability of the approach cannot be questioned. However, the conservatism of the obtained solutions is a valid concern. In order to illustrate some issues, consider a simple example with a $\{2,3\}$-partitioned system matrix. We consider systems with state-space matrices $A^i$, $B$, $C^j$ and $D = 0$.
\begin{gather*}
A^{i} = \left[ \begin{array}{c|c}
A_{11}^i & A_{12} \\ \hline A_{21} & A_{2 2}
\end{array}\right],
A_{11}^1 =  \begin{bmatrix}
- 60 & 30 \\ 20 & -50
\end{bmatrix}, \,A_{1 1}^2 = A_{11}^1/4, \,
A_{12} = \begin{bmatrix}
6 & 6 & 5\\ 0 & 3 &1
\end{bmatrix}, 
A_{21} =  \begin{bmatrix}
4 & 2 \\ 7 &-5 \\ -1 & 1
\end{bmatrix}, \\
A_{22} =\begin{bmatrix}
-90 & 20 & 20 \\ 0 &-10 & 5 \\ -1 & 1 & 50
\end{bmatrix},
B =\begin{bmatrix} 3 & 2 & 5 & 1 &0 \end{bmatrix}^\transpose,\\
C^1 = \begin{bmatrix} 2 & 1 & 5 & 1 & 2 \end{bmatrix}, 
C^2 = \begin{bmatrix} -2 & 1 & 5 & 1 & 2 \end{bmatrix}.
\end{gather*}

One can easily verify that the corresponding comparison systems are all stable. Therefore, Theorem~\ref{prop:comp-ma-hinf} guarantees the existence of a block-diagonal solution to the Riccati inequality for all these systems. Now, we compute several estimates on the $\Hinf$ norm using block-diagonal solutions to the Riccati inequalities. In particular, we consider the following optimisation programs:
\begin{align}
\label{prog:bd_hinf2}\delta_{\rm bd} &= \min\limits_{P, \delta} \delta \\
\notag \text{subject to:  }      & P\succ 0 \text{ is $\alpha$-diagonal and satisfies~\eqref{ric-hinf}},
\end{align}
\begin{align}
\label{prog:bd_sdp}\delta_{\rm sdp} &= \min\limits_{P_i, \Xi_{k i}, \Gamma_{i l}, \Phi_{i j},\delta}\delta \\
\notag \text{subject to:  }      &  P_i\succ 0, \Xi_{k i}, \Gamma_{i l}, \Phi_{i j}\succeq 0 \text{ satisfy~\eqref{ric:dis-dis-test}},
\end{align}
\begin{align}
\label{prog:bd_sc}\delta_{\rm scal} &= \min\limits_{P_i, \xi_{k i}, \gamma_{i l}, \phi_{i j},  \delta}\delta \\
\notag \text{subject to:  }       &  P_i\succ 0, \Xi_{k i}, \Gamma_{i l}, \Phi_{i j}\succeq 0 \text{ satisfy~\eqref{ric:dis-dis-test}}, \\
\notag                           & \Xi_{k i} = \xi_{k i} I, \Gamma_{i l} =  \gamma_{i l} I , \Phi_{i j} =  \phi_{i j} I.
\end{align}

Table~\ref{tab:errors} lists the values $\delta_{\rm bd}$, $\delta_{\rm sdp}$, $\delta_{\rm scal}$, as well as $\| H (s I - F)^{-1} G\|_\Hinf$, normalised by the $\Hinf$ norm of the system. The results in Table~\ref{tab:errors} indicate that the values provided by the program~\eqref{prog:bd_hinf2} are significantly less conservative than the comparison system approach. However, by employing the program~\eqref{ric:dis-dis-test} we can bridge the gap between the values $\delta_{\rm bd}$ and $\| H (s I - F)^{-1} G\|_\Hinf$. Note that for some systems the values $\delta_{\rm sdp}$ are equal to $\|C^i (sI -A^j)^{-1} B\|_\Hinf$, while for others the gap between $\delta_{\rm bd}$ and $\delta_{\rm sdp}$ is not substantial. The SDPs~\eqref{ric:dis-dis-test} are of lower dimensions than~\eqref{ric-hinf}, which can be potentially exploited by distributed optimisation methods. However, this work is not trivial and requires further research. We perform additional simulations in Appendix. 

\subsection{Computational Time Comparison} \label{ss:time}

To demonstrate the scalability, we compare the CPU time required for computing a block-diagonal solution to the Riccati equation using 1) the comparison systems and 2) the direct approach using the systems matrices. We assume that $D =0$ and generate the matrices $A$, $B$ and $C$ randomly where $B$ and $C$ are block-diagonal, the block sizes $k_i$ are random integers between $2$ and $5$. For the comparison system approach, we first compute the scalars $\phi_{i j}$, $\gamma_{i l}$, $\eta_{k i}$ $e_i$, $d_i$, $g_k$, $f_l$ as described in the proof of Lemma~\ref{lem:ric-dist}. Then, we compute a positive-definite solution to the following Riccati equation
\begin{gather} \label{eq:RiccatiSolution}
P_i A_{i i} + A_{i i}^\transpose P_i + e_i/d_i I +P_i\left(\sum_{j \in \cI_i} A_{i j} A_{i j}^\transpose / \phi_{i j} + \sum_{l \in \cL_i} B_{i l} B_{i l}^\transpose/\gamma_{i l}\right) P_i = 0.
\end{gather}
According to Theorem~\ref{prop:comp-ma-hinf}, the solution to~\eqref{eq:RiccatiSolution} gives us a block-diagonal solution $P = \diag{P_1, \dots, P_n}$ to the Riccati inequality for the original system.

Figure~\ref{fig:resp_sdd} shows the computational time required for various number of subsystems $N$.
We note that all the computations for the comparison system approach can easily be parallelised, while the direct approach does not generally allow for parallelisation. As shown in Figure~\ref{fig:resp_sdd}, even without parallelisation, the comparison system approach scales significantly better with the number of subsystems taking at worse $77.3$ seconds to compute versus $92.2$ minutes for the direct approach. The computational results are obtained using Sedumi~\cite{Sedumi} and YALMIP~\cite{YALMIP} on a 4-core Intel i7 3GHz processor with 16GB of RAM. Due to heavy numerical computations for $200$ subsystems we compute every $\Hinf$ norm only once, however, the overall trend in these curves did not significantly change when we repeated the computations.

\begin{figure}[t]
	\centering
	\includegraphics[width = 0.5\columnwidth]{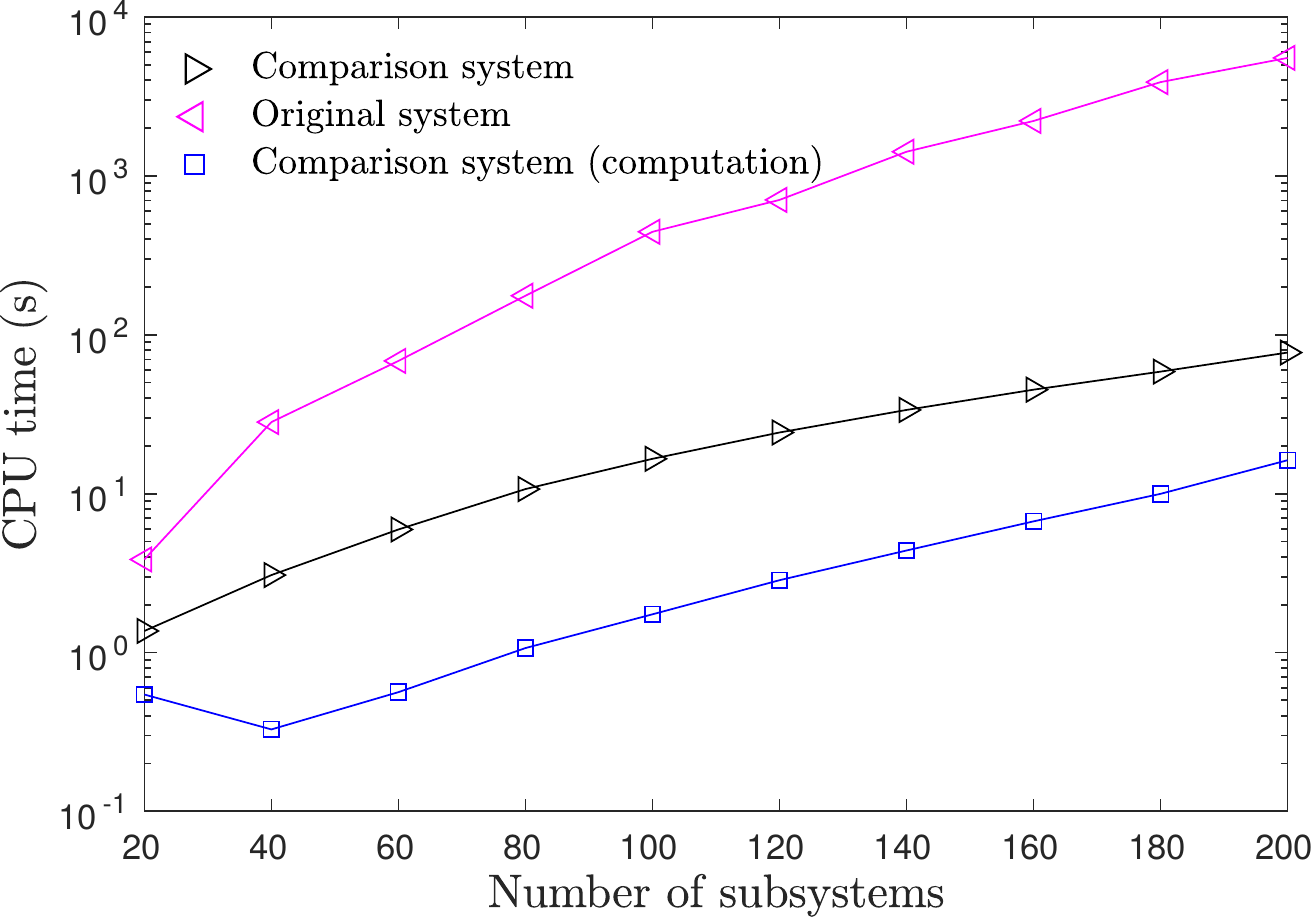}
	\caption{{Computational time comparison in Subsection~\ref{ss:time}. The black and magenta curves with triangle markers depict the total computational time for the comparison system and direct approaches, respectively. The black curve includes the computational time for the composition of the comparison system itself and the matrices $P_i$. The blue curve with square markers depicts the computational time for the Riccati equations and the scalars $e_i$, $d_i$, $g_k$, $f_l$, which are obtained using a linear program.}}
	\label{fig:resp_sdd}
\end{figure}

\subsection{Distributed Stability Tests} \label{add:stab-test}

In~\cite{sootla2017blocksdd}, it was proposed to use the comparison matrices and related Riccati equation in~\eqref{ric:dis-dis-test} to derive stability tests. In this note, we generalise these tests and let
\begin{gather*}
\Phi_{i,-i} = \diag{\begin{bmatrix}
	\phi_{i, 1} I& \cdots &\phi_{i, i-1} I &\phi_{i, i+1} I & \cdots &\phi_{i, n} I
	\end{bmatrix}},
\end{gather*}
where $\phi_{i j}$ with $i,j=1,\dots,n$ and $i\ne j$ are nonnegative scalars. The diagonal elements $\phi_{ii}$ are defined as follows:
\begin{gather*}
\phi_{i i} = \| (sI - A_{i i})^{-1} A_{i,-i} \tilde \Phi_{i,-i}^{-1/2}\|_\Hinf^{-2} +\varepsilon,
\end{gather*}
where $\varepsilon>0$, $\tilde \Phi_{i,-i}$ is obtained by removing all blocks from $\Phi_{i,-i}$ with with $j\not \in \cI_i$. Now we compose the matrix $F = \{\phi_{ij}\}_{i,j=1}^n$. If the matrix $F$ is Hurwitz then the matrix $A$ is Hurwitz and $\alpha$-diagonally stable. Note that stability of $F$ can also be verified in a distributed manner using the conditions in Proposition~\ref{prop:pos-stab} (see \cite{rantzer2015ejc} for details). Our stability tests are obtained by choosing appropriately $\phi_{i j}$; we present a few ad-hoc choices.

{\bf Test I.} The matrix $\cM^\alpha(A)$ is Hurwitz, which implies that $\phi_{i j} = \|(sI - A_{i i})^{-1} A_{i j}\|_{\Hinf} e_i/d_j$ for $j\ne i$, $\phi_{i i} = 1$ and the positive scalars $d_i$, $e_i$ are such that $\sum_{j = 1, j\ne i}^n \|(sI - A_{i i})^{-1} A_{i j}\|_{\Hinf} d_j < d_i$ and $\sum_{j = 1, j\ne i}^n \|(sI - A_{j j})^{-1} A_{j i}\|_{\Hinf} e_j < e_i$.

{\bf Test II.} $\phi_{i j} = \|(s I - A_{i i})^{-1} A_{i j}\|_\Hinf$ for $j\ne i$.

{\bf Test III.}  $\phi_{i j} = \overline{\sigma}(A_{i j})$ for $j\ne i$.

{\bf Test IV.} For $ j\ne i$
$\phi_{i j} = \begin{cases}
1 & \overline\sigma(A_{i j}) > 0, \\
0 & \overline\sigma(A_{i j}) = 0.
\end{cases}$

In Tests~I and~II, we need to compute at most $O(n^2)$ $\Hinf$ norms, while in Tests~III and~IV we need to compute at most $O(n^2)$ matrix norms and $n$ $\Hinf$ norms.
We note that if $\alpha = \{k_1,k_2\}$ then all the tests are equivalent. Note that there is no contradiction with~\cite{sootla2017blocksdd} since the presented tests are not equivalent to the tests in~\cite{sootla2017blocksdd}. However, in general, the set of matrices satisfying  Test~I,~II,~III~and~IV intersect without inclusions, which we show by providing examples. Consider the matrices $A_I$, $A_{II}$, $A_{III}$ and $A_{IV}$:
\begin{align*}
A_I &=\left[\begin{array}{cc|cc|cc}
-2 &    6 &    6  &   2 &    0  &   2\\
0  &  -8  &  -5   & -4  &   1   &  0\\
\hline
2  &  -1  & -12   & -8  &   0   &  2\\
1  &  -1  &  -5   & -6  &   1   &  1\\
\hline
0  &   1  &  -1   &  0  & -11   & -7\\
0  &   1  &   1   & -2  &  -9   &-10
\end{array}\right], \; & & 
A_{II} =\left[\begin{array}{cc|cc|cc}
-4 &    2 &   -1 &   -1  &   0 &   -1\\
9  & -16  &   3  &   8   & -1  &  -1\\
\hline
1  &  -1  &  -3  &  -1   &  1  &  -2\\
-1  &   1 &    4 &   -2  &  -2 &    1\\
\hline
-1  &   2 &    0 &    1 &   -9 &    4\\
-2  &   2 &   -1 &    0 &   -3 &   -4
\end{array}\right],\\
A_{III} &=\left[\begin{array}{cc|cc|cc}
-5  &   3 &   -1 &   -1 &   -1&    -1\\
9   &-14 &    8   &  1 &   -1 &    0\\
\hline
2   & -1 &   -7   & -7 &    0 &    1\\
1   & -1 &    4   & -9 &   -1 &    2\\
\hline
1   & -1 &    1   &  0 &    0 &    4\\
0   & -1 &   -1   &  1 &   -4 &   -5
\end{array}\right], \;
&& A_{IV} = \left[\begin{array}{cc|cc|cc}
-9 &    7&    -3 &    3&     1&     2\\
-6 &   -4&     2 &   -3&    -1&     0\\
\hline
-1 &   -1&    -2 &    5&     1 &    0\\
2  &   2 &   -4  &  -4 &   -2  &   1\\
\hline
2  &  -1 &    0  &   3 &   -9  &  -4\\
0  &   2 &    0  &   0 &    2  &  -7
\end{array}\right].
\end{align*}
Every matrix satisfies the test with the corresponding letter and fails the other ones. For instance, matrix $A_I$ satisfies Test~I and fails Test~II,~III and~IV. We note that it was harder to generate matrices failing either of Tests~I and~II, and satisfying any other test.

Note that all the tests also guarantee that a block-diagonal solution to Lyapunov inequality exists and can be constructed using Riccati equations similar to~\eqref{ric_actual}. Similar ideas can be used to derive other algebraic conditions for the existence of block-diagonal solutions to $\Hinf$ Riccati inequalities.

\section{Conclusion and Discussion} \label{s:con}
In this paper, we have considered a comparison system approach to the analysis of a class of systems. The comparison system is positive and can have a much lower dimension than the original system. If the comparison system is stable, then we can guarantee several strong properties of the original system: the existence of block-diagonal solutions to Lyapunov, Riccati inequalities; efficient, but conservative estimates of norms. {The gap between the set of systems admitting block-diagonal solutions to Lyapunov and Riccati inequalities and the set of systems satisfying our sufficient conditions is not entirely clear. In fact, a similar gap is not characterised in the more studied diagonal case either and still constitutes an interesting theoretical question. Nevertheless, our conditions are relatively easy to verify as they require only linear algebra and linear programming methods.} We present several examples illustrating our theoretical work. We provide additional, but minor theoretical results, as well as additional numerical examples in Appendix. 

Our definition of the comparison matrix results (under additional assumptions) in a decomposition of the Lyapunov and Riccati inequalities into a set of smaller LMIs. This decomposition is similar to the decomposition obtained by chordal sparsity~\cite{mason2014chordal, zheng2016admmdual}. However, there are a few crucial differences between these decompositions. The major one being that chordal decomposition cannot be applied to dense matrices, while diagonal dominance can. On the other hand, chordal decomposition provides necessary and sufficient conditions for the LMI to hold, which is not the case with diagonal dominance. Nevertheless, such a connection can potentially be exploited to derive efficient computational algorithms for large-scale system analysis using the techniques in~\cite{zheng2016admmdual,chordalZheng2018ecc} and the recently proposed methods in~\cite{sootla2019block, zheng2019block}.

Finally, some generalisations of our results were not discussed in the note due to their triviality. Instead of the $2$-norms in the definition of comparison matrix $\cM^\alpha(A)$, one can also use other $p$-induced norms for off-diagonal terms. However, we would need to compute the $p$-induced norms of the transfer matrices obtained by the matrices on the block-diagonal, which by itself is not an easy task in general unless $\alpha = \bfone$,  or the blocks $A_{ii}$ are Metzler.

\bibliography{biblio_bd}

\balance
\appendix

\section{Relating Comparison Systems}\label{app:comp-matrices}

In this appendix we present an explicit relationship between our new comparison matrix~\eqref{comp-sys-hard}, the optimisation-based comparison matrix obtained from Lemma~\ref{lem:ric-dist}, and the comparison matrix from~\cite{sootla2017blocksdd}.

\begin{defn}\label{def:block-comp-old}
	Given an $\alpha$-partitioned matrix $A$ with Hurwitz $A_{i i}$, we define the matrix $\widetilde\cM^\alpha(A)$ as follows:
    \begin{equation}
	   \widetilde \cM^\alpha_{ij}(A) = \begin{cases} -\|(sI - A_{i i})^{-1}\|_\Hinf^{-1} &  \text{if }i = j, \\
	 \|A_{i j}\|_2& \textrm{otherwise}.
	\end{cases} \label{block-comparison-old}
	\end{equation}
\end{defn}

Our second comparison matrix comes from contraction theory for nonlinear systems~\cite{russo2013contraction}.
\begin{defn}\label{def:block-comp-2}
	Given an $\alpha$-partitioned matrix $A$ with Hurwitz $A_{i i}$, we define the matrix $\cN^\alpha(A)$ as follows
	\begin{equation}
	\cN^\alpha_{ij}(A) = \begin{cases} \min\{\mu_2(A_{i i}), 0\} &  \text{if }i = j, \\
	\|A_{i j}\|_2 & \textrm{otherwise},
	\end{cases} \label{block-comparison-2}
	\end{equation}
	where $\mu_2(X) = \lim_{h\rightarrow 0+} \frac{1}{h} \left( \| I_k + h X \|_2 - 1\right) = \overline \lambda (X + X^\transpose)/ 2$ and $X\in\R^{k \times k}$, where $\overline\lambda(Z)$ denotes the maximum eigenvalue of a symmetric matrix $Z$.
\end{defn}

We also consider the class of matrices $A$ satisfying {Lemma~\ref{lem:ric-dist}, that is, the matrices} for which there exist $P_i\succ 0$, $\Psi_{i j}, \Phi_{i j}\succeq 0$ satisfying the following constraints:
\begin{subequations}
	\label{block-sdd-ric}
	\begin{align}
	&P_i A_{i i} + A_{i i}^\transpose P_i +\Phi_{i i} + \Psi_{i i} \prec 0~\forall i,\\
	& \begin{bmatrix}
	\Psi_{i j}       & P_i A_{i j} \\
	A_{i j}^\transpose P_i & \Phi_{j i}
	\end{bmatrix} \succeq 0~\forall i\ne j,\\
	&\Phi_{i i} \succ \sum\limits_{j = 1, j \ne i}^n \Phi_{j i},\,\,\Psi_{i i} \succ \sum\limits_{j = 1, j \ne i}^n \Psi_{i j}~\forall i.
	\end{align}
\end{subequations}

In order to formalise how these comparison matrices are related, we define the following classes of matrices
\begin{align*}
\cC^\alpha_0 &= \{A\in \R^{N\times N} \Bigl| \cN^\alpha(A) \text{ is Hurwitz} \},\\
\cC^\alpha_1 &= \{A\in \R^{N\times N} \Bigl| \widetilde \cM^\alpha(A) \text{ is Hurwitz} \}, \\
\cC^\alpha_2 &= \{A\in \R^{N\times N} \Bigl| \cM^\alpha(A) \text{ is Hurwitz} \},\\
\cC^\alpha_3 &= \{A\in \R^{N\times N} \Bigl| \exists P_i\succ 0, \Psi_{i j}, \Phi_{i j}\succeq 0: A \text{ satisfies~\eqref{block-sdd-ric}} \},\\
\cC^\alpha_4 &= \{A\in \R^{N\times N} \Bigl| A \text{ is $\alpha$-diagonally stable} \}.
\end{align*}
\begin{prop}\label{prop:comp-mat-rel} We have the following inclusions:
	\begin{gather*}
	\cC^\alpha_0 \subseteq
	\cC^\alpha_1 \subseteq \cC^\alpha_2 \subseteq \cC^\alpha_3\subseteq \cC^\alpha_4.
	\end{gather*}
	for a given $\alpha$. If $\alpha = \bfone$, then 
	 $\cC^\alpha_0 = \cC^\alpha_1 = \cC^\alpha_2 = \cC^\alpha_3$.
\end{prop}
\begin{proof}(i) Since $\widetilde \cM^\alpha(A)$ is Hurwitz there exist positive scalars $d_i$ such that

\begin{gather*}
\|(s I - A_{i i})^{-1}\|_\Hinf^{-1} d_i > \sum\limits_{j = 1, j \ne i}^n \| A_{i j}\|_2 d_j.
\end{gather*}
Now it is straightforward to get:
\begin{equation}\label{ineq:comp-matr-inclusion}
    \begin{aligned}
    d_i > \sum\limits_{j = 1, j \ne i}^n \|(s I - A_{i i})^{-1}\|_\Hinf \| A_{i j}\|_2 d_j 
       \ge  \sum\limits_{j = 1, j \ne i}^n \|(s I - A_{i i})^{-1} A_{i j}\|_\Hinf  d_j,
 \end{aligned}
\end{equation}
therefore $\cM^\alpha(A)$ is Hurwitz.

(ii) The inclusions $\cC^\alpha_2 \subseteq \cC^\alpha_3\subseteq \cC^\alpha_4$ follow from Theorem~\ref{prop:comp-ma-hinf}.

(iii) Since $\cN^\alpha(A)$ is Metzler and Hurwitz, then $\mu_2(A_{i i}) < 0$. We can show that with  $P = -\mu_2(A_{i i}) I$, the following inequality holds:
\begin{gather*}
P A_{i i} + A_{i i}^\transpose P + (\mu_2(A_{i i}))^2 I + P^2 \preceq 0.
\end{gather*}
Indeed, using the inequality $A_{i i} + A_{i i}^\transpose \preceq 2\mu_2(A_{i i}) I$, we get
\[
\begin{aligned}
-(A_{i i} + A_{i i}^\transpose)\mu_2(A_{i i})  + (\mu_2(A_{i i}))^2 I + (\mu_2(A_{i i}))^2 I 
\preceq - 2(\mu_2(A_{i i}))^2 I  + (\mu_2(A_{i i}))^2 I + (\mu_2(A_{i i}))^2 I 
= 0.
\end{aligned}
\]

Therefore, $\|(s I - A_{i i})^{-1}\|_\Hinf \le  -1/\mu_2(A_{i i})$, where the equality is attained, \emph{e.g.}, for scalar $A_{i i}$. This leads to $\cN^\alpha(A) \ge \widetilde \cM^\alpha(A)$. Since $\cN^\alpha(A)$ is Hurwitz, there exists $v \gg 0$ such that $\cN^\alpha(A) v \ll 0$. Due to $\cN^\alpha(A) \ge \widetilde \cM^\alpha(A)$ the same vector $v$ can be used to show stability of $\widetilde \cM^\alpha(A)$.

{(iv) In the trivial partitioning case,  $\cN^\bfone(A) = \widetilde \cM^\bfone(A)$. Furthermore, the nonstrict inequality in~\eqref{ineq:comp-matr-inclusion} becomes an equality. Therefore, if the matrix $\cM^\bfone(A)$ is Hurwitz, so is the matrix $\widetilde \cM^\bfone(A)$.}
\end{proof}

Since in the trivial partitioning case the comparison matrices are equivalent, the comparison matrix $\cN^\alpha(A)$ can also be seen as a block-generalisation of scaled diagonal dominance.

Finally, using Proposition~\ref{prop:comp-mat-rel} we have the following implication from Theorem~\ref{prop:comp-ma-hinf}.
\begin{cor} \label{Cor:Comparsions}
	The results of Lemma~\ref{lem:ric-dist} and Theorem~\ref{prop:comp-ma-hinf} hold if we set $\widetilde F = \widetilde \cM^\alpha(A)$
	(or $\widetilde F = \cN^\alpha(A)$)
	and $\widetilde G_{i l} = \|B_{i l}\|_2$, while $H$ and $J$ defined in the same way. In particular, if
$$\|H(s I - \widetilde F)^{-1} \widetilde G + J\|_\Hinf <\delta, $$
 then $\|C(s I-A)^{-1} B + D\|_\Hinf <\delta$ and there exist $P_i \succ 0$ such that~\eqref{ric-hinf} holds with $P = \diag{P_1, \dots, P_n}$.
\end{cor}

\section{Bounds on the Outputs and States}\label{app:cs_n}
Using the matrix $\cN^\alpha$, we can define another comparison system
\begin{equation}
F = \cN^\alpha(A ), G_{i l} = \|B_{i l} \|_2,
H_{k j} = \| C_{k j} \|_2, J_{k l} = \|D_{k l}\|_2.\label{comp-sys-simple}
\end{equation}
As we will show below the comparison matrix $\cN^\alpha$ provides additional relations between the original system~\eqref{eq:system-full} and the comparison system~\eqref{eq:system-comparison}.

\subsection{Main Results}
As shown in Corollary~\ref{Cor:Comparsions}, the result of Theorem~\ref{prop:comp-ma-hinf} applies to the case of a comparison system with the matrix $\cN^\alpha$.
Moreover, this case opens the door of exploiting additional properties of the comparison systems. For instance, it is possible to bound the state and the output of system~\eqref{eq:system-full} using its comparison system.

\begin{thm} \label{thm:comp-prin}
	Consider system~\eqref{eq:system-full} with $x(0)= x^0$ and its comparison system~\eqref{eq:system-comparison} defined in~\eqref{comp-sys-simple} with $F= \cN^\alpha(A )$, $\xi_i(0) = \xi^0_i = \| x_i^0\|_2$ and $\upsilon_l(t) = \|u_l(t)\|_2$ with $i=1,\dots, n$, $l=1,\dots,n_i$. Then for all $i=1,\dots, n$, $k=1,\dots,n_o$:
	\begin{gather*}
	\|x_i(t)\|_2 \le \xi_i(t),\,\,\,\| y_k(t) \|_2 \le \nu_k(t), \,\,\forall t \ge 0.
	\end{gather*}
\end{thm}
\begin{proof}
	For all $i$, $x_i \in \R^{k_i}$, $x_j\in \R^{k_j}$ and  a small $\varepsilon>0$, we have 	
	\begin{equation}\label{ineq:comp-block}
	\sum\limits_{j = 1}^n x_i^\transpose A_{i j} x_j + \sum\limits_{l =1}^{n_i} x_i^\transpose B_{i l} u_l < (F_{i i} + \varepsilon) \|x_i \|_2^2 
	 + \sum\limits_{j = 1, j \ne i}^n F_{i j} \|x_i\|_2 \|x_j\|_2 + 	\sum\limits_{l = 1}^{n_i} G_{i l} \| x_i\|_2 \|u_l\|_2,
	\end{equation}
	where we use the bounds
$$z^\transpose X z = z^\transpose (X + X^\transpose) z/2 \le -\mu_2(X) z^\transpose z = -\mu_2(X) \|z\|_2^2,$$
and
$$z^\transpose Y y \le \|Y\|_2 \|z\|_2 \|y\|_2 $$
for all vectors $z$, $y$ and matrices $X$, $Y$ of appropriate dimensions.

\begin{equation} \label{prop4:contr-cond1}
    \begin{aligned}
	\frac{1}{2} \frac{d (\xi_i^\varepsilon(t))^2}{dt } &= (F_{i i} + \varepsilon )\|x_i(t)\|_2^2 + \sum\limits_{j = 1, j\ne i}^n  F_{i j} \xi_i^\varepsilon(t) \xi_j^\varepsilon(t)  +
	\sum\limits_{l = 1}^{n_i} G_{i l} \xi_i^\varepsilon(t) \upsilon_l   \\
&= (F_{i i} + \varepsilon ) \|x_i(t)\|_2^2 +\sum\limits_{j = 1, j\ne i}^n F_{i j} \| x_i(t)\|_2 \xi_j^\varepsilon(t) +  \sum\limits_{l = 1}^{n_i} G_{i l} \| x_i(t)\|_2 \|u_l\|_2 \\
 &\ge
	(F_{i i} + \varepsilon ) \| x_i(t)\|_2^2
	+ \sum_{j = 1, j\ne i}^n F_{i j} \| x_i(t)\|_2 \|x_j(t)\|_2 + 	\sum\limits_{l = 1}^{n_i} G_{i l} \| x_i(t)\|_2 \|u_l\|_2.
	\end{aligned}
\end{equation}

Consider now the system $\dot \xi = (F+\varepsilon I) \xi + G \upsilon$ and let $\xi^\varepsilon(t)$ denote its solution with $\xi_i^\varepsilon(0) = \|x_i^0\|_2$. Using~\eqref{ineq:comp-block}, we get
\[
\frac{d \|x_i(t)\|_2^2}{d t}\Bigl|_{t =0} < \frac{d(\xi_i^\varepsilon(t))^2}{d t} \Bigl|_{t =0},
\]
which means that there exists $T>0$ such that $\|x_i(t)\|_2 < \xi_i^\varepsilon(t)$ for all $t\in[0,~T]$ and all $i$.
	
	Let there exist $s$ such that $\|x_j(\sigma)\|_2 < \xi_j^\varepsilon(\sigma)$ for all $\sigma \in [0, s)$ and all $j$, however,
	$\|x_j(s)\|_2 \le  \xi_j^\varepsilon(s)$	for all $j$ and there exists an index $i$ such that $\|x_i(s)\|_2 =  \xi_i^\varepsilon(s)$. This implies that
	\begin{equation}\label{prop4:contr-cond}
	0\ge\frac{d ( (\xi_i^\varepsilon(t))^2 - \| x_i(t)\|_2^2) }{d t}\Bigl|_{t = s}.
	\end{equation}
	On the other hand, the inequality shown in~\eqref{prop4:contr-cond1} leads to
\[
	\frac{1}{2}\frac{d ((\xi_i^\varepsilon(t))^2 - \|x_i(t)\|_2^2)}{d t}\Bigl|_{t = s}  >0.
\]	
	We arrived to the contradiction with~\eqref{prop4:contr-cond}. Therefore, $\xi_i^\varepsilon(t) > \| x_i(t)\|$ for all $t> 0$ and all $i$. Letting $\varepsilon\rightarrow 0$ we get $\xi_i(t) \ge \| x_i(t)\|_2$ for all $t> 0$ and all $i$.
	
	Now we will show the second part of the statement. According to triangle and Cauchy-Schwartz inequalities we have
\[
    \begin{aligned}
	\| y_k(t) \|_2 =  \left\| \sum\limits_{j = 1}^{n} C_{k j} x_j(t) + \sum\limits_{l = 1}^{n_i} D_{k l} u_l \right\|_2  
                  & \le \sum\limits_{j = 1}^n \left\| C_{k j}\right\|_2 \left\|x_k(t)\right\|_2  + \sum\limits_{l = 1}^{n_i} \|D_{k l}\|_2 \|u_l\|_2  \\
                  &\le 	\sum\limits_{j = 1}^n H_{k j} \xi_j + \sum\limits_{l = 1}^{n_i} J_{k l} \upsilon_l 
                  = \nu_k(t),
	\end{aligned}
\]
which completes the proof.
\end{proof}

Using the flow bounds we can use stability analysis tools for positive systems such as $\dot \xi = F \xi$ to study nonpositive systems such as $\dot x = A x$.

\begin{cor} \label{cor:lyap-fun}
	Consider the system~\eqref{eq:system-full} and its comparison system~\eqref{eq:system-comparison} defined in~\eqref{comp-sys-simple} with $F= \cN^\alpha(A)$.	Let $F$ be Hurwitz and let the vectors $d$, $e\in\Rp^n$ be such that  $-F d$, $ - e^\transpose F \in \Rp^n$, then
	\[
    \begin{aligned}
	   V_m(x) = \max\limits_{i = 1,\dots, n} \{ \|x_i\|_2/d_i\},\quad
       V_s(x) = \sum\limits_{i = 1}^n e_i \|x_i\|_2, \quad 
       V_d(x) = \sum\limits_{i = 1}^n e_i/d_i \|x_i\|_2^2,
	\end{aligned}
\]
	are Lyapunov functions for $\dot x = Ax$.
\end{cor}
\begin{proof}
	According to Theorem~\ref{thm:comp-prin}, with $\xi_i(t) = \|x_i(t)\|_2$, $\dot x = A x$ and $\dot \xi = F \xi$ we have that
	$\| x_i(s)\|_2 \le \xi_i(s)$ for all $s>t$. Let $W(\xi) = \max\limits_{i = 1,\dots, n} \{ \xi_i/d_i\}$, then we have
	\begin{gather*}
	V_m(x(s)) \le W(\xi(s)) < W(\xi(t)) = V_m(x(t)),
	\end{gather*}
	for all $s>t$. This implies that
\[
    \begin{aligned}
    	\dot V_m(x) = \liminf_{ h \searrow 0} \frac{V_m(x(t+h)) - V_m(x(t))}{h} 
       \le \liminf_{ h \searrow 0} \frac{W(\xi(t+h)) - W(\xi(t))}{h} 
       <0.
    \end{aligned}
\]
Therefore, $V_m(x)$ is a valid Lyapunov function with $\dot V_m(x)<0$ in the points of differentiability. The cases of $V_s$ and $V_d$ \yang{are} treated similarly.
\end{proof}

\subsection{Relation to Network Input-to-State Stability}
We will show that the Lyapunov function $V_d$ has an input-to-state stability (ISS) interpretation.  Consider a fully observable system:
\begin{gather}
\label{eq:system-full-1}
\begin{aligned}
\dot{x}(t)&=A x(t) + B u(t),
\end{aligned}
\end{gather}
and its comparison system $\dot \xi = F \xi + G \upsilon $, where $F = \cN^\alpha(R A R^{-1})$, $G_{i l} = \|R_i B_{i l}\|_2$ for some invertible $\alpha$-diagonal $R = \diag{R_1, \dots, R_n}$. Using the inequality~\eqref{ineq:comp-block} in Theorem~\ref{thm:comp-prin} and completion of squares technique that given $P_i = e_i/d_i R_i^\transpose R_i$, $\gamma_{i l} = G_{i l} e_i/g_l$, $\phi_{i j}= |F_{i j}|e_i/e_j$ with positive $e_i$, $d_i$, $g_l$ satisfying~\eqref{cond-pos-hinf-1} and \eqref{cond-pos-hinf-2}, we can obtain the following inequalities:
\begin{equation} 
\label{cond:iss-na}
    \begin{gathered}
        \sum\limits_{j = 1}^n x_i^\transpose P_i A_{i j} x_j  + \sum_{l =1}^{n_i} x_i^\transpose P_i B_{i l} u_l \le - \phi_{i i} x_i^\transpose P_i x_i  +
        \sum\limits_{j =1, j\ne i}^n \phi_{i j} x_j^\transpose P_j x_j + \sum\limits_{l =1}^{n_i} \gamma_{i l} u_l^\transpose u_l,\,\,\forall x_i, u_l, \\
    \sum_{j = 1, j \ne i}^{n} \phi_{j i} < \phi_{i i}.
      \end{gathered}
\end{equation}

Such systems are said to satisfy ISS small gain conditions, see~\cite{dashkovskiy2010small} and the references within. Stability of the interconnected system is shown using a comparison system with $\hat F_{i j} = \phi_{i j}$ and $\hat G_{i l} = \gamma_{i l}$. Construction of max- and sum-separable functions will also follow from the ISS conditions. However, in general, the relationship between the flows of the full and comparison systems can be more complicated than the ones described in Theorem~\ref{thm:comp-prin}. Furthermore, the linear case is considered in~\cite{dashkovskiy2010small} and only a nonlinear comparison system was derived. Therefore, our results preserve linearity of comparison systems, which is beneficial in the linear case. The matrix $\cM^\alpha(A)$ can also be used to derive ISS-type small gain conditions, however, in this case we cannot derive a linear comparison system as in the case of $\cN^\alpha$, where the state-space transformation $R$ is the key to build $P$.

{Condition~\eqref{ric:dis-dis-test} from Lemma~\ref{lem:ric-dist} can be used to derive a similar to~\eqref{cond:iss-na} set of condition. In particular, Lemma~\ref{lem:ric-dist} implies that there exist $P_i\succ 0$, $\Xi_{k i}, \Gamma_{i l}, \Phi_{i j}, \Upsilon_{k l}, \Lambda_{k l}\succeq 0$ such that for all $x_i\in \R^{k_i}$, $u_l \in \R^{m_l}$, we have:}
\begin{subequations}\label{cond:iss-ma}
\begin{gather}
\sum\limits_{j = 1}^n x_i^\transpose P_i A_{i j} x_j \le - x_i^\transpose \Phi_{i i} x_i  + 
\sum\limits_{j =1, j\ne i}^n  x_j^\transpose \Phi_{i j} x_j + \sum\limits_{l =1}^{n_i}  u_j^\transpose \Gamma_{i l} u_l,\label{cond:iss-ma-1}\\
\sum_{j = 1, j \ne i}^{n} \Phi_{j i} \preceq \Phi_{i i}.\label{cond:iss-ma-2}
\end{gather}
\end{subequations}

Note that the right hand side in~\eqref{cond:iss-ma-1} does not depend on the Lyapunov functions $x_i^\transpose P_i x_i$, which makes conditions~\eqref{cond:iss-ma-2} conceptually different from the conditions~\eqref{cond:iss-na} as well as the conditions in~\cite{cook1974stability, araki1975application, dashkovskiy2010small} (for nonlinear systems). We note that the conditions~\eqref{cond:iss-na}  (as well as the conditions in~\cite{cook1974stability, araki1975application, dashkovskiy2010small}) require optimisation over the gains $\gamma_{i l}$,$\phi_{i j}$, and storage functions $x_i^\transpose P_i x_i$. These optimisation programs to our best knowledge are typically non-convex, our approach, on the other hand, leads to polynomial time algorithms.

\section{Application to Dissipative Networks} \label{app:diss}
In the context of input-output behaviour, \emph{dissipativity} is considered as a typical analysis tool, which is defined with the help of \emph{the storage function} $V(x) = x^\transpose P x$ and \emph{the supply rate} $w(y,u)$:
\begin{gather*}
2 x^\transpose P (A x + B u) \le -w(y, u) = - \begin{bmatrix}
y \\ u
\end{bmatrix}^\transpose W \begin{bmatrix}
y \\ u
\end{bmatrix}, \forall x, u, y,
\end{gather*}
where $P$ is a positive semidefinite matrix and $W$ is symmetric:
\begin{gather*}
W = \begin{bmatrix}
W_{1 1} &   W_{1 2} \\
W_{1 2}^\transpose &  -W_{2 2}
\end{bmatrix}.
\end{gather*}
In particular, if we set $W_{1 2} = 0$, $W_{1 1} = I$, $W_{2 2} = \delta^2 I$, then we can estimate the $\Hinf$ norm of the system using \emph{the Bounded Real Lemma}~\cite{ZDG}.

One can see the partitioned system~\eqref{eq:system-full} as an interconnection of $n$ systems, where the Hurwitz matrices $A_{i i}$ model the dynamics of individual subsystems, while the terms $A_{i j}$ for $i \ne j$ model their interconnections. There are
a number of ways to model an interconnection of linear systems. For example, consider the setting in Figure~\ref{fig:network-interconnection}, where $\cG_i = \cC_i (s I - \cA_{i})^{-1} \cB_i$ and the constant matrices $M$, $K$ and $N$ are partitioned according to the inputs and the outputs to $\cG_i$. Let the diagonal block entry $M_{i i}$ equal to zero, meaning that we forbid direct feedback loops. This setting was considered in~\cite{meissen2015compositional} and the conditions on dissipativity of the network were derived using local dissipativity conditions, \emph{i.e.}, dissipativity of the subsystems. Assume that the rows of the mapping $\left[\begin{smallmatrix} z \\ y \end{smallmatrix}\right] \rightarrow w_i$ are linearly independent. Let the local dissipativity conditions for subsystems with the supply rate defined by $Y_i$ and the storage function $x_i^\transpose \cP_i x_i$:
\begin{gather}
\begin{gathered}
\begin{bmatrix}
\cP_i \cA_i + \cA_i^\transpose \cP_i  + \cC_i^\transpose Y^i_{1 1} \cC_i & \cP_i \cB_i + \cC_i^\transpose Y^i_{1 2} \\
\ast & -Y^i_{2 2}
\end{bmatrix} \preceq 0, \\
Y^i = (Y^i)^\transpose  = \begin{bmatrix}
Y^i_{1 1 } &  Y^i_{ 1 2} \\
(Y^i_{1 2})^\transpose & -Y^i_{2 2}
\end{bmatrix},
\end{gathered} \label{local-supply-cond}
\end{gather}
where $\ast$ stands for the transpose of the upper right corner of the matrix.

\begin{figure}[t]
	\centering
	\includegraphics[width = 0.3\columnwidth]{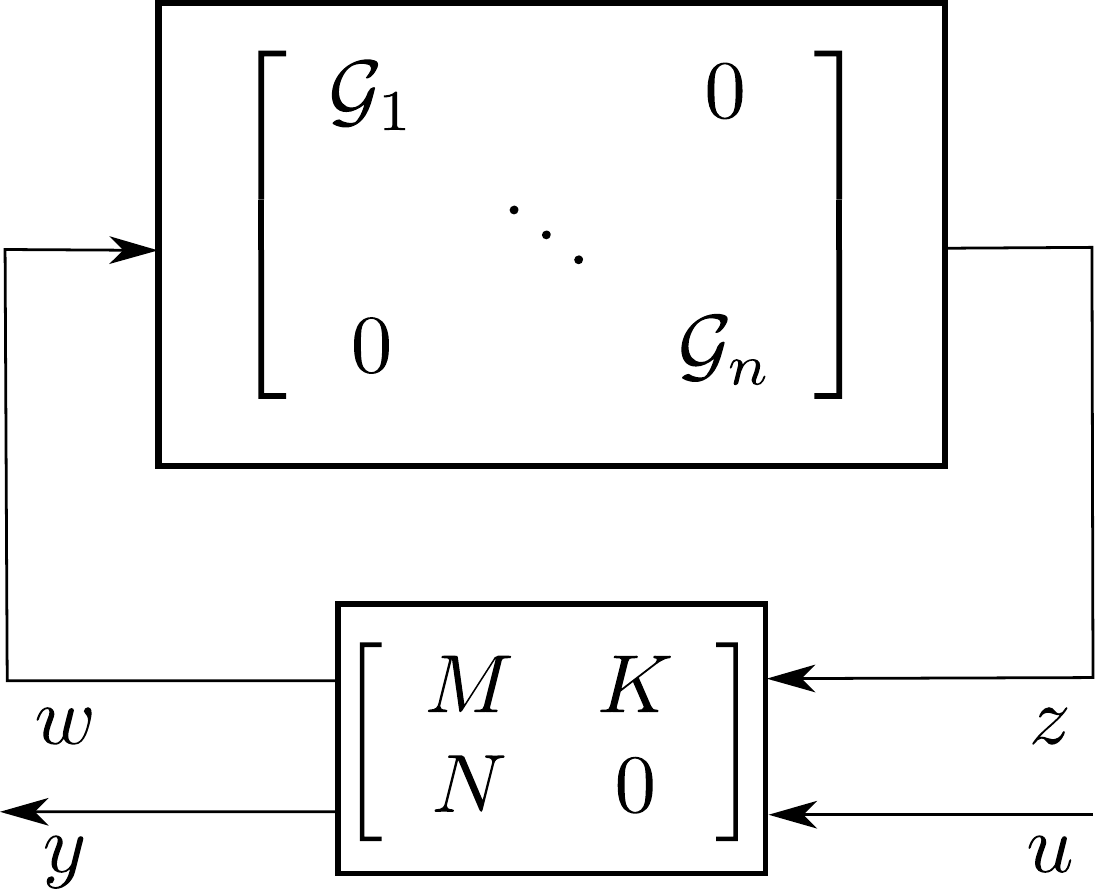}
	\caption{Network of systems modelled by a feedback interconnection}
	\label{fig:network-interconnection}
\end{figure}
Consider the centralised coupling constraint
\begin{gather} \label{global-supply-cond}
\begin{bmatrix}
M^\transpose & N^\transpose & I &0 \\ K^\transpose & 0 & 0 & I
\end{bmatrix}
T_\pi^\transpose Q T_\pi
\begin{bmatrix}
M^\transpose & N^\transpose & I &0 \\ K^\transpose & 0 & 0 & I
\end{bmatrix}^\transpose
\preceq 0,
\end{gather}
where $Q = \diag{-Y^1, \dots, -Y^n, W}$, $W = W^\transpose$ is fixed in advance and specifies the global supply rate, and the matrix $T_\pi$ is a permutation matrix such that
\[
\begin{bmatrix}
w_1^\transpose & z_1^\transpose & \cdots & w_n^\transpose & z_n^\transpose & u^\transpose & y^\transpose
\end{bmatrix} =   \begin{bmatrix}
w^\transpose & u^\transpose & z^\transpose & y^\transpose
\end{bmatrix} T_\pi^\transpose.
\]
In particular, if $Y_{1 2}^i$ and $W_{1 2}$ are zero matrices, then $T_\pi^\transpose Q T_\pi = \diag{\cY_2, W_{1 1},  -\cY_1,  -W_{2 2}}$, where
$\cY_1 = \diag{Y^1_{1 1}, \dots, Y^n_{1 1}}$ and $\cY_2 = \diag{Y^1_{2 2}, \dots, Y^n_{2 2}}$. One of the main result in~\cite{meissen2015compositional} states that under some mild assumptions~\eqref{local-supply-cond} and~\eqref{global-supply-cond} hold if and only if the network is dissipative with a storage function $\sum_{i =1}^n x_i^\transpose \cP_i x_i$ with respect to the supply rate specified by $W$. These conditions can only be verified with semidefinite programming with the decision variables $\cP_i\succ 0$, $Y^i = (Y^i)^\transpose$.

Our framework can also be applied to this case. One can set $\cA_{i} = A_{i i}$, $A_{i j} = \cB_i M_{i  j} \cC_j$ for $i \ne j$, while $C_{k j} = N_{k j}\cC_j$ and $ B_{i l} = \cB_{j} K_{i l}$ with $i, j = 1,\dots, n$, $k = 1,\dots, n_o$, $l = 1,\dots, n_i$ and simply apply the theory to the resulting system. However, we consider the following comparison system
\begin{equation} \label{comp-sys-sg}
\begin{aligned}
F_{i j} &= \begin{cases}
-{1} & \text{if~} i = j,\\
\|\cC_i (s I - \cA_i)^{-1} \cB_i M_{i j}\|_\Hinf &  \text{otherwise,}
\end{cases} \\
G_{i l} &= \|\cC_i (s I - \cA_i)^{-1} \cB_i K_{i l}\|_\Hinf,\,\, H_{k j} = \overline{\sigma}(N_{k j}).
\end{aligned}
\end{equation}
In order to decouple the system information from interconnection information, one can use:
\begin{equation} \label{comp-sys-sg-1}
\begin{aligned}
\widetilde F_{i j} &= \begin{cases}
-\|\cC_i (s I - \cA_i)^{-1} \cB_i\|_\Hinf^{-1} & \text{if~} i = j,\\
\|M_{i j}\|_2 &  \text{otherwise,}
\end{cases} \\
G_{i l} &= \overline{\sigma}(K_{i l}), \,\, H_{k j} = \overline{\sigma}(N_{k j}),
\end{aligned}
\end{equation}
which would lead to more conservative estimates. A positive aspect of these representations is the independence on the state-space representations of subsystems $\cG_i$. The following result is the consequence of applying our techniques in this setting.

\begin{prop}
	\label{prop:small-gain}
		Consider the network of the stable subsystems $\cG_i = \cC_i (s I - \cA_i)^{-1} \cB_i$ (the matrices $\cA_i$ are Hurwitz) interconnected through matrices $M$, $N$ and $K$ as in Figure~\ref{fig:network-interconnection}. Consider the comparison system~\eqref{eq:system-comparison} with the state-space matrices defined in~\eqref{comp-sys-sg} and let there exist positive vectors $e$, $d$, $g$, $f$, and a scalar $\delta$ satisfying~(\ref{cond-pos-hinf-1},\ref{cond-pos-hinf-2}). Then 

(i) the network also satisfies conditions~(\ref{local-supply-cond},~\ref{global-supply-cond}), with
$Y_{1 1}^{i} = \phi_{i i} I\succ 0$, $Y_{2 2}^i\succ 0$ $Y_{1 2}= 0$, $W_{1 1} = I$, $W_{2 2} = -\delta^2 I$, and $W_{1 2} = 0$.  

(ii)  there exists $P = \diag{P_1, \dots, P_n}$ such that
			\begin{gather*}
			P A + A^\transpose P + C^\transpose C + P B B^\transpose P/\delta^2 \prec 0
			\end{gather*}
			
and $\|C (s I - A)^{-1} B \|_\Hinf < \delta$.
\end{prop}
\begin{proof}
	(i) We will only sketch the proof due to similarity to the proof of Theorem~\ref{prop:comp-ma-hinf}. We set $\eta_{k j} = H_{k j} d_j/g_k$, $\gamma_{i l} = G_{i l} e_i/ f_l$, $\phi_{i j} = F_{i j} e_i/d_j$, where the scalars $e_i$, $d_i$, $f_l$, $g_k$ satisfy~(\ref{cond-pos-hinf-1},\ref{cond-pos-hinf-2}). We define $\tilde \Phi_i$, $\tilde \Gamma_i$, $\cI_i$, $\cL_i$, $\cK_i$ as in Lemma~\ref{lem:ric-dist}, and obtain the bounds
		\begin{equation}\label{small-gain-cond}
		\begin{aligned}
		&\|\cC_i (s I - \cA_{i})^{-1} \cB_i\begin{bmatrix} M \tilde\Phi_i^{-1/2} & N \tilde \Gamma_i^{-1/2}\end{bmatrix}\|_\Hinf^2 < \phi_{i i}^{-1},\\
		&\sum_{j \in \cI_i} \phi_{j i} I + \sum_{k \in \cK_i} N_{k i}^\transpose N_{k i} /\eta_{k i}  \prec \phi_{ii} I.
		\end{aligned}
		\end{equation}

Using these relations, we get the following Riccati inequalities with coupling constraints:
	\begin{gather}
	\label{ric-l2gain-network} P_i \cA_{i} + \cA_{i}^\transpose P_i + \cC_i^\transpose Y^i_{1 1} \cC_i  + P_i \cB_i (Y^i_{2 2})^{-1} \cB_i^\transpose P_i \prec 0,\\
	\label{ric-l2gain-network-1}Y^i_{1 1}      \succ \sum\limits_{j \in \cI_i}\Phi_{j i} + \sum\limits_{k \in \cK_i}N_{k i}^\transpose \Xi_{k i}^{-1} N_{k i}, \,\,\,\\
	\label{ric-l2gain-network-2}(Y^i_{2 2})^{-1} \succ \sum\limits_{j \in \cI_i} M_{i j} \Phi_{i j}^{-1} M_{i j}^\transpose + \sum\limits_{l \in \cL_i}K_{i l} \Gamma_{i l}^{-1} K_{i l}^\transpose,  \\
	\Gamma_{i l}, \Phi_{i j}, \Xi_{k i} \succeq 0, \,\sum_{i =1}^{n} \Xi_{k i} \preceq I, \,\, \sum_{i =1}^{n} \Gamma_{i l} \prec \delta^2 I,
	\end{gather}	
	where $Y_{1 1}^i = \phi_{i i} I$, $\Phi_{i j}=\phi_{i j} I$, $\Xi_{k i} = \eta_{k i} I$, and $\Gamma_{i l} = \gamma_{i l}I$. The conditions in~\eqref{ric-l2gain-network-1} imply that $\cY_1  \succ \hat \Phi + N^\transpose  N$,	where $\hat \Phi = \diag{\sum_{j\ne 1, j =1}^n\Phi_{j 1}, \cdots, \sum_{j\ne n, j = 1}^n\Phi_{j n}}$, $\cY_1 = {\rm diag}\{Y_{1 1}^1,$ $\dots, Y_{11}^n\}$. While the conditions in~\eqref{ric-l2gain-network-2} imply that
		\begin{equation}
		\cY_2^{-1}  \succ  M \hat \Phi^{-1} M^\transpose + K \left(\sum\limits_{j = 1}^{n_i} \hat \Gamma_i\right)^{-1} K^\transpose 
		\succ  M \hat \Phi^{-1} M^\transpose + K \delta^{-2} K^\transpose, \label{ineq:sg-1}
		\end{equation}
	where $\hat \Gamma_i = \diag{\sum_{j =1}^n\Gamma_{j 1}, \cdots, \sum_{j = 1}^n\Gamma_{j n_i}}$, $\cY_2 = {\rm diag}\{Y_{2 2}^1,$ $\dots, Y_{2 2}^n\}$. Applying Schur's complement twice to~\eqref{ineq:sg-1} and using $\cY_1  \succ \hat \Phi + N^\transpose N$ yields the following chain of inequalities
	\begin{equation}
	\cY_1 - N^\transpose W_1 N \succ \hat \Phi \succ  \begin{bmatrix}
	M \\ 0
	\end{bmatrix}^\transpose \begin{bmatrix}
	\cY_2^{-1} & K \\
	K^\transpose     & W_2
	\end{bmatrix}^{-1} \begin{bmatrix}
	M \\ 0
	\end{bmatrix}, \label{cond:suff}
	\end{equation}
	where $W_1 = I$, $W_2 = \delta^2 I$. Using Schur's complement again we get:
	\begin{gather*}
	\begin{bmatrix}
	\cY_1     &  N^\transpose   & M^\transpose     & 0 \\
	N         &  W_1^{-1} & 0          & 0 \\
	M         &   0       & \cY_2^{-1} & K \\
	0         &   0       & K^\transpose     & W_2
	\end{bmatrix} \succ 0 \Leftrightarrow  
	-\begin{bmatrix}
	\cY_1 & 0 \\
	0     & W_2
	\end{bmatrix} + \begin{bmatrix}
	M & N \\ K & 0
	\end{bmatrix}^\transpose \begin{bmatrix}
	\cY_2 & 0 \\
	0     & W_1
	\end{bmatrix} \begin{bmatrix}
	M & N \\ K & 0
	\end{bmatrix} \prec 0,
	\end{gather*}
	which after some algebraic manipulations leads to~\eqref{global-supply-cond} with the matrix $Q$ specified above.
	
	(ii) the proof is straightforward.
\end{proof}

As with our previous results, the $\Hinf$ norm of the network can be estimated using linear programming or algebra. The constraints~(\ref{local-supply-cond},~\ref{global-supply-cond}) with $Q$ described in Proposition~\ref{prop:small-gain} are necessary of stability of the comparison system. For sufficiency, at least the existence of an $\alpha$-diagonal $\tilde \Phi$ satisfying~\eqref{cond:suff} is required. On the other hand the constraints~(\ref{ric-l2gain-network}-\ref{ric-l2gain-network-2}) can be relaxed by replacing scalar variables with positive semidefinite matrices.
Note that all the constraints can be transformed to convex ones using Schur's complement. Finally, assume that the systems $\cG_i$ are single-input-single-output and such that $\cC_i\|(sI -\cA_{i})^{-1} \cB_i\|_\Hinf = \|-\cC_i \cA_i^{-1}\cB_i\|_2$, for example all subsystems are internally positive. In this case the constraint~\eqref{small-gain-cond} is tight, that is for any valid choice of $\phi_{i j}$, $\gamma_{i j}$ we can set $\phi_{ii}^{-1}$ to be equal to $\|\cC_i(sI - \cA_{i})^{-1} \cB_i\begin{bmatrix} M \tilde\Phi_i^{-1/2} & N \tilde \Gamma_i^{-1/2}\end{bmatrix}\|_\Hinf^2$.

\section{Additional Examples}

\begin{figure*}
	\centering
	\includegraphics[height = 0.28\columnwidth]{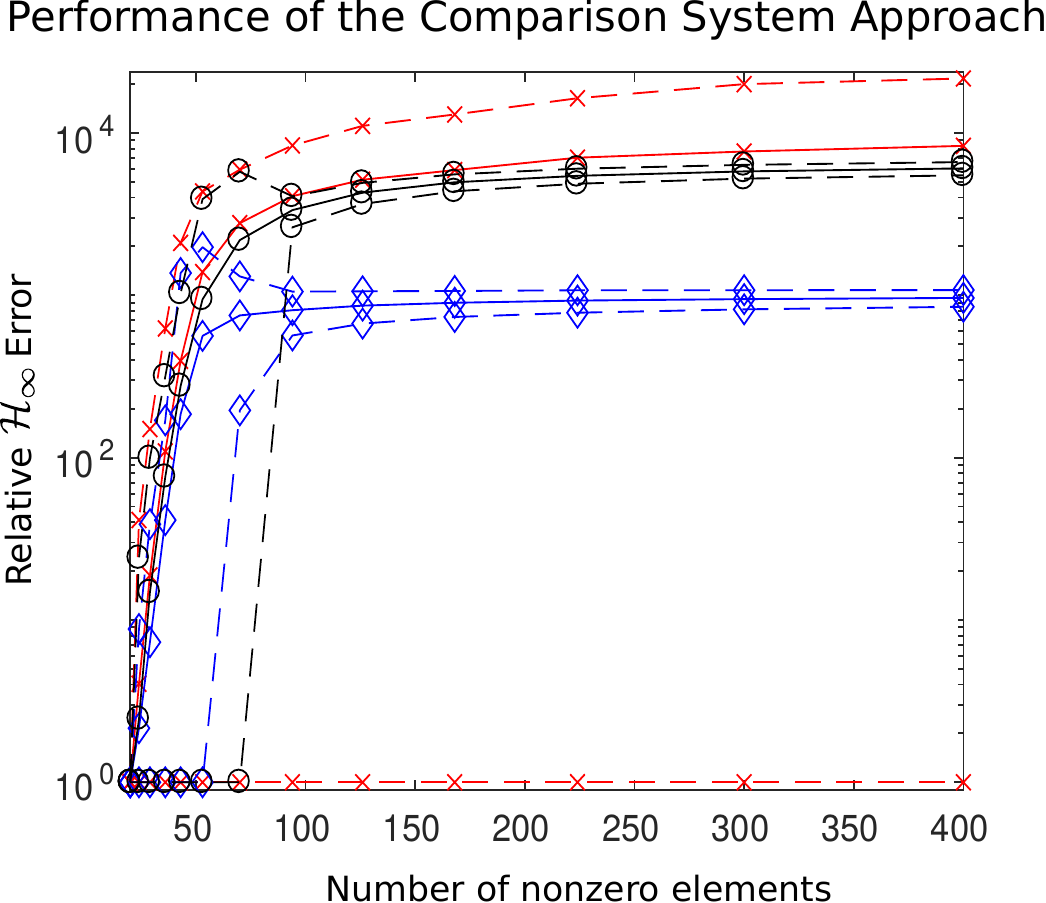} 
\includegraphics[height = 0.28\columnwidth]{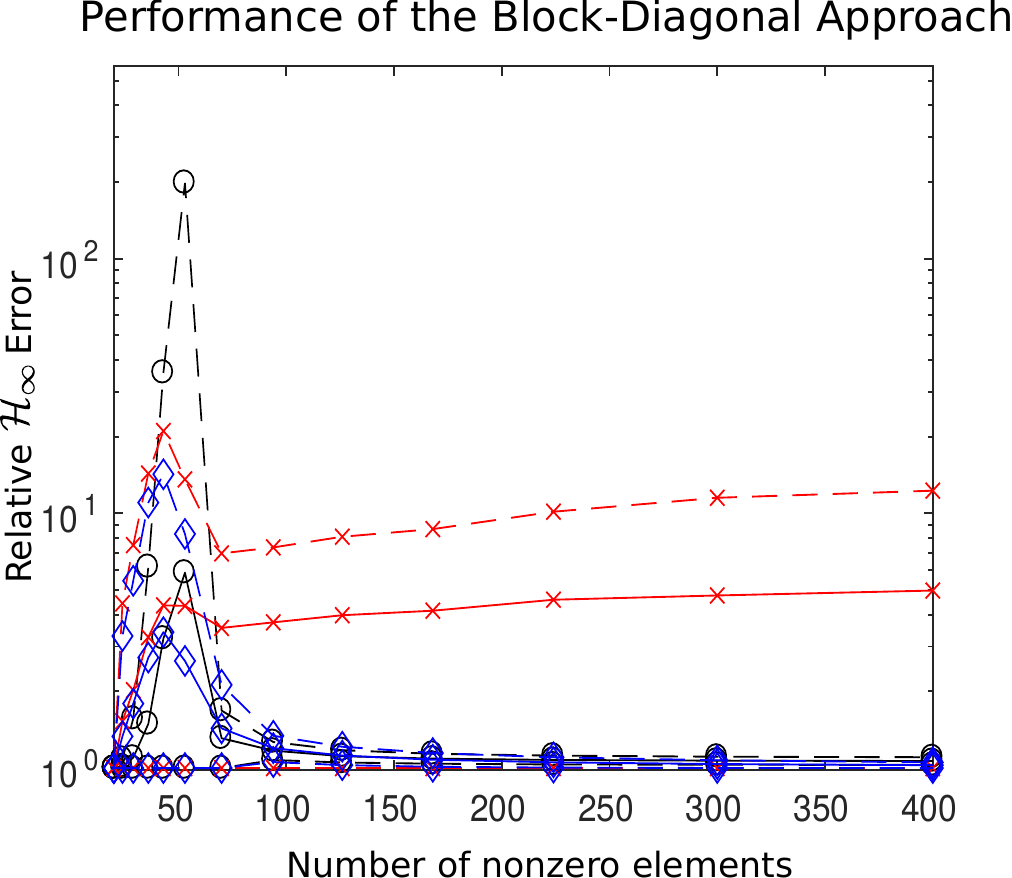} 
\includegraphics[height= 0.28\columnwidth]{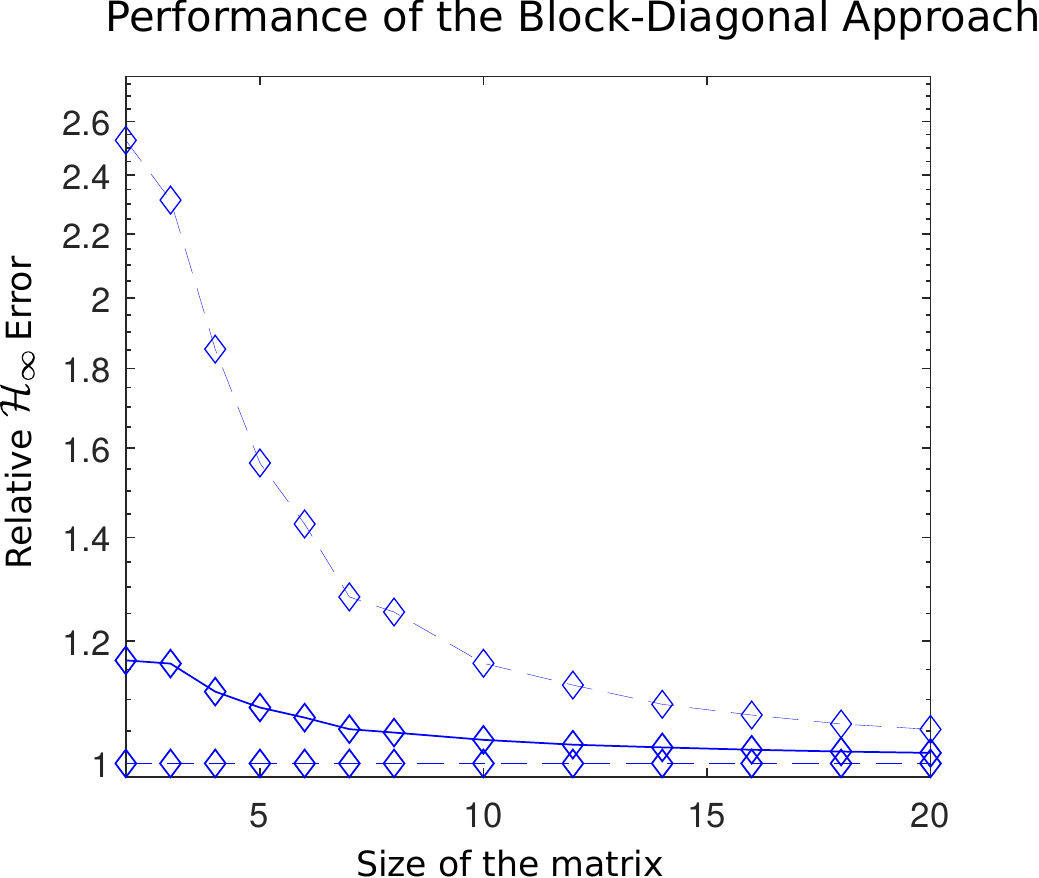}
	\caption{Mean values (solid lines) and $90\%$ confidence intervals (dashed lines). The left panel: the values $\delta_{\rm cs}^1$ (red lines with cross markers), $\delta_{\rm cs-p}^1$ (blue lines with diamond markers), and $\delta_{\rm cs}^{20}$ (black lines with circle markers) against the number of nonzero elements in $20\times20$ matrices. The centre panel: the values $\delta_{\rm bd}^1$ (red lines with cross markers), $\delta_{\rm bd-p}^1$ (blue lines with diamond markers), and $\delta_{\rm bd}^{20}$ (black lines with circle markers) against the number of nonzero elements in $20\times20$ matrices. The right panel: the values $\delta_{\rm bd}^1$ against the size of the full matrices. }
	\label{fig:perf_44}
\end{figure*}

\subsection{\texorpdfstring{$\Hinf$}{H infinity} Performance Analysis with \texorpdfstring{$\alpha = \bfone$}{a = 1}}
We consider random systems with $A\in\R^{20\times 20}$ such that $\cM^\bfone(A)$ is Hurwitz, $B^k\in\R^{20\times k}$, $C^k\in\R^{k\times 20}$ and $D = 0$.
We first generate a random matrix $\tilde A$ with the fixed number of nonzero entries where all the nonzero entries are distributed according to the uniform distribution $\cU$ on the interval $[0,~1]$. We then define $\cM^\alpha(A) = \tilde A - ((1 + \varepsilon)\max(0, \max_i(\Re(\lambda_i(\tilde A)))))I$, where $\lambda_i(\tilde A)$'s are eigenvalues of $\tilde A$ and $\varepsilon$ is a scalar larger than zero. After that we obtain the matrix $A$ by flipping with probability $0.5$ the signs of the nondiagonal entries of the matrix $\cM^\alpha(A)$. We generate full matrices $B^k$ and $C^k$ with entries distributed according to $\cU(-1, 1)$ for two cases $k= 1$ and $k = 20$. The dominant eigenvalue of $\cM^\bfone(A)$ will lie close to the origin, but similar results were observed if the spectrum is shifted farther to the left. We then compute the following quantities
\[
\delta_{\rm bd-p}^k= \frac{\delta_{\rm bd-p}^k}{\| \cG_{p}^k\|_\Hinf},\,\, \delta_{\rm bd}= \frac{\delta_{\rm bd}^k}{\| \cG^k \|_\Hinf},
\delta_{\rm cs-p}^k= \frac{\| \cG_{cs}^k\|_\Hinf}{\| \cG_{p}^k\|_\Hinf},\,\, \delta_{\rm cs}= \frac{\| \cG_{cs}^k\|_\Hinf}{\| \cG^k\|_\Hinf},
\]
where $\cG^k = C^k (s I -A)^{-1} B^k$, $\cG_{cs}^k = H^k (s I - F)^{-1} G^k$, $\cG_{p}^k = |C^k| (s I -A)^{-1} |B^k|$, where $|X|$ stands for the entry-wise application of the absolute value function to the matrix $X$, while $\delta_{\rm bd}^k$ and $\delta_{\rm bd-p}^k$ stand for solutions of the program
\begin{align}
\label{prog:bd_hinf}\delta_{\rm bd} = \min\limits_{P, \delta}~~& \delta \\
\notag \text{subject to:  }      & P\succ 0 \text{ is $\alpha$-diagonal and satisfies~\eqref{ric-hinf}}
\end{align}
with matrices $A$, $B^k$, and $C^k$ and matrices $A$, $|B^k|$, and $|C^k|$, respectively. We remind the reader that $k$ denotes the dimension of the matrices $B^k$ and $C^k$. Since we compare the relative norms of the systems, we recover the loss of generality by restricting the support of distributions of the entries of the matrices $A$, $B^k$ and $C^k$. That is, similar results are obtained while generating the entries of $\tilde A$ using $\cU([0, a_{\rm max}])$, $B_{i j}^k\sim \cU([-b_{\rm max}, b_{\rm max}])$ and $C_{i j}^k\sim \cU([-c_{\rm max}, c_{\rm max}])$  for some positive $a_{\rm max}$, $b_{\rm max}$, $c_{\rm max}$.

For each number of nonzero elements, we generated $10000$ matrices as described above. The mean values of $\delta_{\rm cs-p}^1$, $\delta_{\rm cs}^1$ and $\delta_{\rm cs}^{20}$ are depicted in the left panel in Figure~\ref{fig:perf_44}, while the mean values of $\delta_{\rm bd}^1$, $\delta_{\rm bd-p}^1$ and $\delta_{\rm bd}^{20}$ are depicted in the centre panel in Figure~\ref{fig:perf_44}. There were two not entirely expected results in these simulations. Firstly, the values $\delta_{\rm bd-p}^1$, $\delta_{\rm bd}^{20}$ converge to the values close to $1$ on average with the number of nonzero entries increasing, while the values $\delta_{\rm bd}^1$ start increasing again with the number of nonzero entries larger than $80$. Secondly, the relative errors $\delta_{\rm bd-p}^1$, $\delta_{\rm bd}^1$, $\delta_{\rm bd}^{20}$ peak between $40$ and $50$ nonzero entries, and later decrease. We do not depict the results for $\delta_{\rm bd-p}^{20}$, $\delta_{\rm cs-p}^{20}$, since $\delta_{\rm bd-p}^{20}$ has a quite similar performance to $\delta_{\rm bd}^{20}$, while $\delta_{\rm cs-p}^{20}$ remarkably has a quite similar performance to $\delta_{\rm cs-p}^1$.

It is not entirely clear why sign-indefinite low-rank matrices $B^k$ and $C^k$ add conservatism to the solution of the Riccati inequalities on average, however, we can elaborate on the conservatism peak for sparse matrices. Let $B^k$ and $C^k$ be diagonal matrices, hence $P B^k (B^k)^\transpose P/\delta^2 + (C^k)^\transpose C^k$ is diagonal. In this case, the conservatism should originate with matrix $A$. If the matrix $A$ is full and scaled diagonally dominant, then the values on the diagonal generally have larger magnitudes than the off-diagonal terms. In contrast to sparse matrices, the off-diagonal elements can have comparable magnitudes with the elements on the diagonal. Therefore rescaling with a diagonal $P$ may not be sufficient to compensate for $P B^k (B^k)^\transpose P/\delta^2 + (C^k)^\transpose C^k$ in the case of sparse matrices more often than in the case of full matrices.

Our observations lead to a conclusion that for sparse scaled diagonally dominant matrices using diagonal matrices $P$ in Lyapunov/Riccati inequalities may not be advisable, even though such $P$ exist. Instead one should use block-diagonal $P$ in order to reduce the conservatism. Naturally, two questions arise: (a) what is the size of the blocks on the diagonal of $P$, and (b) how to choose the pattern of $P$.
Some indications on how to approach these questions are provided in~\cite{mason2014chordal, zheng2016admmdual}, where chordal sparsity can help to determine the size of the blocks of $P$ and the pattern of $P$.
In the case of a), we can provide further analysis. Consider the right panel in Figure~\ref{fig:perf_44}, where we plot the relative error against the size $n$ of the full matrix $A$. Note that the conservatism reduces considerably as $n$ grows and with $n$ approaching $20$ disappears with a high probability. Therefore, it appears that for the block sizes up to $10$, it is beneficial to use similar size blocks in the matrix $P$. However, if the dimension of the full blocks is larger than $20$, then we can employ a sparse $P$ without a significant loss of performance.

We also note that $\delta_{\rm bd-p}^k$ is not always smaller than $\delta_{\rm bd}^k$. For example, set $k =1$ and
\begin{gather*}
A = \begin{bmatrix}
-5 &  0 & 0 \\-7 & -7 & 0 \\ 6  & 3  &-4
\end{bmatrix}\,\,  B^1 = \begin{bmatrix}
4 \\ -4 \\ 1
\end{bmatrix},\,\,  C^1 = \begin{bmatrix}
6 \\ 3 \\ -4
\end{bmatrix}^T.
\end{gather*}
In this case, however, we have that $\delta_{\rm bd-p}^1 < \delta_{\rm bd}^1$, but at the same time $\| \cG_p^1 \|_\Hinf < \| \cG^1 \|_\Hinf$ and $\delta_{\rm bd-p}^1>\delta_{\rm bd}^1$. Therefore, positivity of $B^1$ and $C^1$ is still beneficial for solvability of the LMI in~\eqref{prog:bd_hinf}. We conclude this example by indicating that further studies of the systems with Hurwitz $\cM^\alpha(A)$ can lead to scalable, but less conservative analysis methods than the comparison system approach.

\subsection{Time and Frequency Responses with \texorpdfstring{$\alpha = \bfone$}{a=1}}
\label{ex:flow}

\begin{figure}
	\centering
	\subfigure[ ]{\includegraphics[height = 0.3\columnwidth]{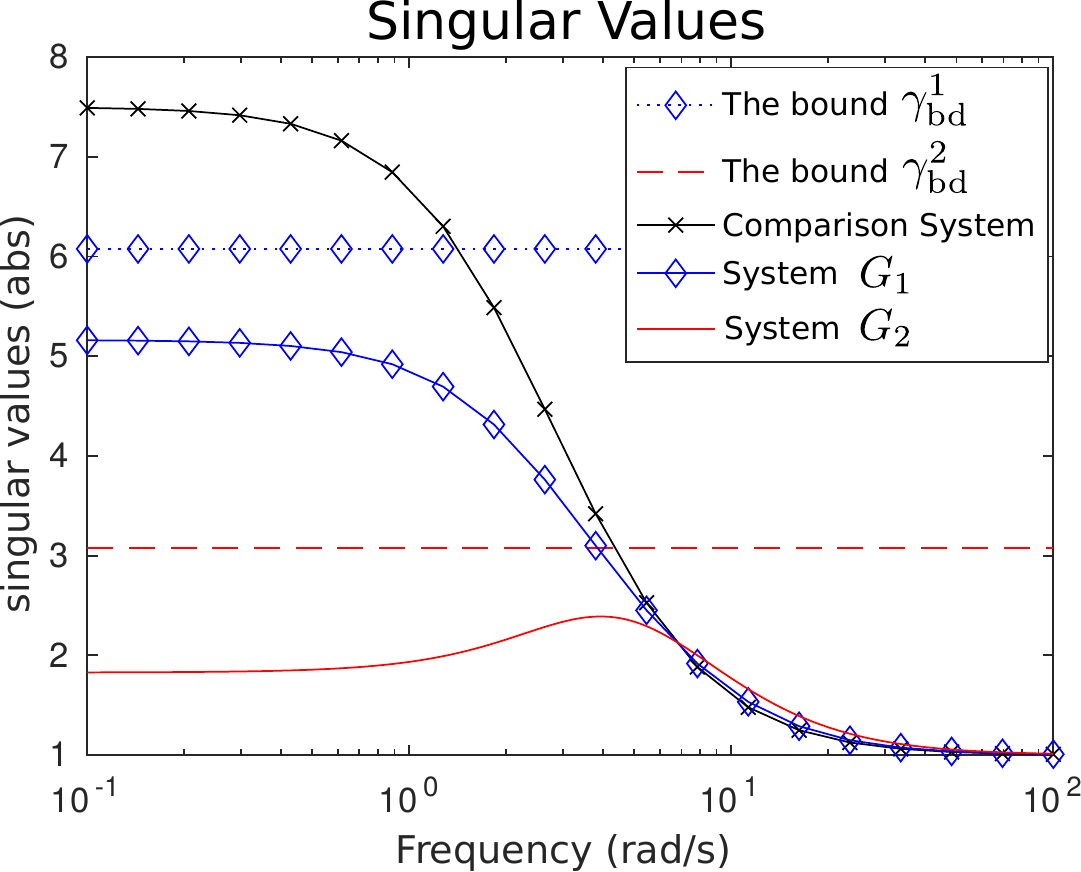}} \hspace{10mm}
	\subfigure[ ]{\includegraphics[height = 0.3\columnwidth]{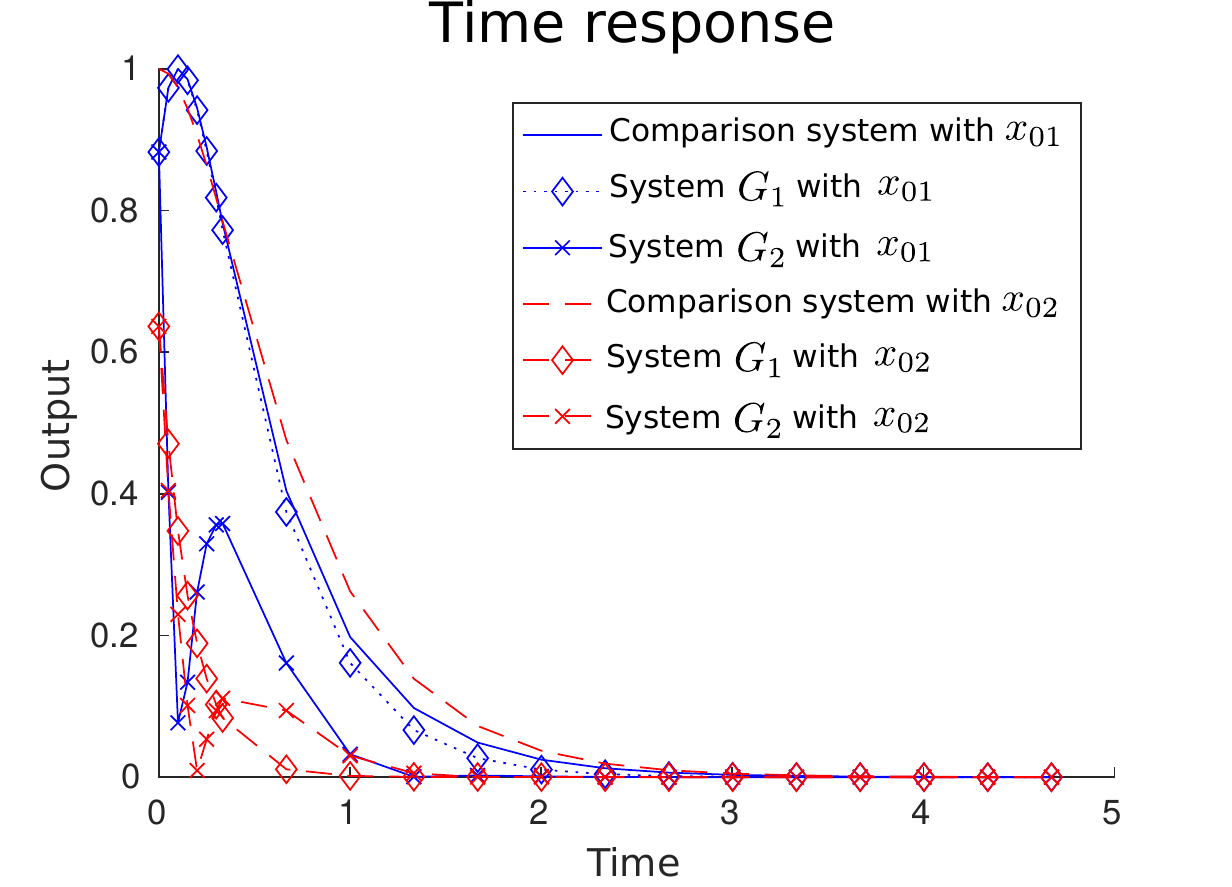}}
	
	\caption{Responses of Systems $G_1$ and $G_2$ in Subsection~\ref{ex:flow}: (a) Frequency response; (b) Initial condition response.}
	\label{fig:resp_sdd2}
\end{figure}
To illustrate the result in Theorem~\ref{thm:comp-prin}, we consider the time and frequency responses of the following system

\begin{gather*}
G_1 = \left[ \begin{array}{c|c}A_1 & B \\ \hline C & D
\end{array}\right] = \left[\begin{array}{ccccc|c}
-5 & -2  & -1 & 0  &  4 & 3 \\
0  & -5  & -3 & -1 &  0 & 0 \\
0  & -2  & -9 & 0  &  0 & 1 \\
0  &  0  & -2 & -5 &  1 & 0 \\
1  &  3  &  0 & 0  & -4 & 1 \\
\hline
1  &  2  &  0 & 0  &  8& 1
\end{array}\right],
\end{gather*}
where  $\cM^\bfone(A_1)$ is Hurwitz and hence a diagonal solution to the Riccati inequality~\eqref{ric-hinf} exists. We also flip a sign of the $(5,1)$ entry of the matrix $A_1$ and get the system $G_2 = \left[ \begin{array}{c|c}A_2 & B \\ \hline C & D \end{array}\right]$. First, we perform a similar analysis as above and compute the norms ant their estimates. In Figure~\ref{fig:resp_sdd2}(a), we plot the singular values of the systems and the bounds  $\delta_{\rm bd}^1$,  $\delta_{\rm bd}^2$, which are the solutions to~\eqref{prog:bd_hinf} for systems $G_1$ and $G_2$. Our main observation is that a flip of a sign drastically changes the magnitude of the frequency response of the systems. Yet the bounds $\delta_{\rm bd}$ are still surprisingly close to the actual $\Hinf$ norm.

Next, we evaluate the flow bounds provided by Theorem~\ref{thm:comp-prin}. We compute the initial condition responses of the systems $G_1$ and $G_2$ and their comparison system to the initial conditions $x_{01} = \begin{bmatrix} -1 & 0 & 0 & 0 & 0 \end{bmatrix}$ and $x_{02} = \begin{bmatrix} -1 & 1 & -1 & 1 & -1 \end{bmatrix}$. As shown in Figure~\ref{fig:resp_sdd2}(b), the initial condition responses (with $u=0$) can be conservative also for the system $G_1$, especially if the initial state does not belong to the orthants $\Rnn^n$ or $-\Rnn^n$, where the bound is tightest.

\end{document}